\documentclass[smallextended]{svjour3}

\usepackage{amsmath}
\usepackage{amsthm}
\usepackage{amssymb}
\usepackage{algorithm,algorithmic}
\usepackage{listings}
\usepackage{color}
\usepackage{bm}
\usepackage{graphicx}
\usepackage{url}
\usepackage{tikz}
\usepackage{tkz-fct}
\usetikzlibrary{calc,arrows}
\usepackage[siunitx]{circuitikz}
\usepackage{pgfplots}
\usepackage{multirow}
\pgfplotsset{compat=1.8}
\usepackage{mathtools} 

\definecolor{codegreen}{rgb}{0,0.6,0}
\definecolor{codegray}{rgb}{0.5,0.5,0.5}
\definecolor{codepurple}{rgb}{0.58,0,0.82}
\definecolor{codeblue}{rgb}{0,0.2,0.8}
\definecolor{backcolour}{rgb}{0.95,0.95,0.92}

\lstdefinestyle{mystyle}{
	backgroundcolor=\color{backcolour},   
	commentstyle=\color{codegreen},
	keywordstyle=\color{codeblue},
	numberstyle=\tiny\color{codegray},
	stringstyle=\color{codepurple},
	basicstyle=\footnotesize,
	breakatwhitespace=false,         
	breaklines=true,                 
	captionpos=b,                    
	keepspaces=true,                 
	numbers=left,                    
	numbersep=5pt,                  
	showspaces=false,                
	showstringspaces=false,
	showtabs=false,                  
	tabsize=2
}

\def\inum#1{\protect{\textrm{\boldmath $#1$}}}
\def\Solset{S(\inum{p})}
\def\bnum#1{\protect{\mathbf{#1}}}
\newcommand{\ints}{\mathbb{I}\reals}
\newcommand{\reals}{\mathbb{R}}
\newcommand{\R}{\reals}
\DeclarePairedDelimiter\braces{\lbrace}{\rbrace}	
\DeclarePairedDelimiter\parentheses{\lparen}{\rparen}	
\DeclareMathOperator{\hull}{hull}	
\def\ihull#1{\hull\parentheses*{#1}}		
	
\newcommand{\mna}[1]{{\mathcal{#1}}}
\def\Mid#1{{#1}^c}              		
\def\Rad#1{{#1}^\Delta}         		
\def\nref#1{$(\ref{#1})$}

\begin{document}

\title{On preconditioning and solving an extended class of interval parametric linear systems\thanks{M.~Hlad\'{\i}k was supported by the Czech Science Foundation Grant P403-18-04735S.}}

\titlerunning{On preconditioning and solving an extended class of interval parametric systems}

\author{Iwona Skalna\footnote{AGH University of Science and Technology, Krak{\'o}w, Poland}\\
Milan Hlad{\'i}k\footnote{Department of Applied Mathematics, 
    Charles University, Prague, Czech Republic}}

\author{Iwona Skalna \and Milan Hlad\'{i}k}
\institute{
  I. Skalna \at
  AGH University of Science and Technology, Poland, \\
  \email{skalna@agh.edu.pl}    
\and
  M. Hlad\'{i}k \at 
    Department of Applied Mathematics, 
    Charles University, 
    Malostransk\'e n\'am. 25, 118 00, 
    Prague, Czech Republic,\\
  \email{hladik@kam.mff.cuni.cz} 
}

\date{Received: date / Accepted: date}

\maketitle

\begin{abstract}
We deal with interval parametric systems of linear equations and the goal is to solve such systems, which basically comes down to finding an~enclosure for a~parametric solution set. Obviously we want this enclosure to be as tight as possible. The review of the available literature shows that in order to make a~system more tractable most of the solution methods use left preconditioning of the system by the midpoint inverse.  Surprisingly, and in contrast to standard interval linear systems, our investigations have shown that double preconditioning can be more efficient than a~single one, both in terms of checking the regularity of the system matrix and enclosing the solution set. Consequently, right (which was hitherto mentioned in the context of checking regularity of interval parametric matrices) and double preconditioning together with the $p$-solution concept enable us to solve a~larger class of interval parametric linear systems than most of existing methods. The applicability of the proposed approach to solving interval parametric linear systems is illustrated by several numerical examples.
\keywords{preconditioning \and revised affine forms \and interval parametric linear systems \and parametric solution}
\end{abstract}


\section{Introduction}
\label{sec:intro}

Solving systems of parametric linear equations with parameters varying within prescribed intervals is an important part of many scientific and engineering computations. The reason is that (parametric) linear systems are prevalence in virtually all areas of science and engineering, and uncertainty is a~ubiquitous aspect of most real world problems. 

Consider the following family of systems of parametric linear equations
\begin{equation}
\label{eq:parintlinsys}
\{A(p)x=b(p),\,p\in\inum{p}\},
\end{equation}
where $A(p)\in\reals^{n\times n}$, $b(p)\in\reals^n$, and $\inum{p}$ is a~$K$-dimensional interval vector. The entries of $A(p)$ and $b(p)$ are assumed, in general case, to be real-valued continuous functions\footnote{They usually have closed form expressions.} of a~vector of parameters $p$, i.e.,
\[
A_{ij},b_i\colon \mathbb{R}^K\rightarrow\mathbb{R},\ i,j=1,\ldots,n.
\]

A~particular form of \nref{eq:parintlinsys} arises when there are affine-linear dependencies. This means that the entries of $A(p)$ and $b(p)$ depend linearly on $p=(p_1,\dots,p_K)$, that is, $A(p)$ and $b(p)$ have, respectively, the following form:
\[
A(p)=A^{(0)}+\sum_{k=1}^KA^{(k)}p_k,\quad b(p)=b^{(0)}+\sum_{k=1}^Kb^{(k)}p_k.
\]

The family (\ref{eq:parintlinsys}) is often written as
\begin{equation}
\label{eq:pils}
A(\inum{p})x=b(\inum{p})
\end{equation}
to underline its strong relationship with interval linear systems \cite{Neumaier:1990:IMS}. Indeed, a~classical nonparametric $n\times n$ interval linear system $\inum{A}x=\inum{b}$ can be considered as a~special case of an~interval parametric linear system with $n(n+1)$ interval parameters. However, \textit{interval parametric linear systems} (IPLS) usually provide more precise results, especially for the real-life problems involving uncertainties. That is why numerical methods for solving such systems are of great value.

The solution set of the system (\ref{eq:pils}) can be defined in many ways, however, usually the so-called \textit{united parametric solution set} is considered, which is defined as the set of solutions to all systems from the family (\ref{eq:parintlinsys}), i.e., 
\begin{equation}
\label{eq:unitedsolset}
\Solset\triangleq\{x\in\reals\mid\exists\,p\in\inum{p}\;:\;A(p)x=b(p)\}.
\end{equation}
Since handling $S(\inum{p})$ is, in general, computationally very hard\footnote{The results on the complexity of various problems related to interval matrices and interval linear systems can be found, e.g., in \cite{HorHla2017a,KreLak1998,Rohn:HILP}; since an interval matrix (interval linear system) can be treated as a~special case of an~interval parametric matrix (interval parametric linear system) with each parameter occurring only once, all these results are valid for interval parametric matrices and interval parametric linear systems, too.}, instead of the solution set itself one usually seeks for an interval vector that encloses $\Solset$.
A~more general approach to the problem of solving IPLS was developed by Kolev \cite{Kolev:2016:IAFD,Kolev:2014:PSL}. He introduced a~new type of solution, the so-called \textit{parametric solution} (or shortly $p$-\textit{solution}), which has the following parametric form
\begin{equation}
\inum{x}(p)=Fp+\inum{a},
\end{equation}
where $F\in\reals^{n\times K}$ and $\inum{a}$ is an $n$-dimensional one column vector. The main advantage of the $p$-solution over the classical interval solution is that it preserves information about linear dependencies between $x=x(p)$ and $p$ (see Fig.~\ref{fig:psol_vs_intsol}).

\pgfplotsset{compat=1.11}
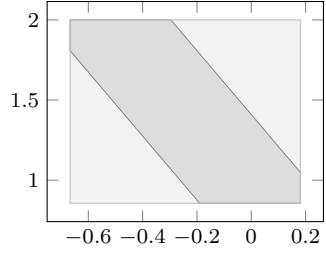
\begin{figure}
\centering
\begin{tikzpicture}
\begin{scope}[xshift=8.25cm]
\begin{axis}[height=4.5cm]
\addplot[gray,fill=gray,fill opacity=0.2] coordinates {
	(-0.666906517,2.000252053)
	(-0.666906517,1.807677268)
	(-0.18983027,0.85570584)
	(-0.004442137,0.85570584)
	(0.180945996,0.85570584)
	(0.180945996,1.048280625)
	(-0.296130251,2.000252053)
	(-0.481518384,2.000252053)
	(-0.666906517,2.000252053)
};
\addplot[lightgray,fill=lightgray,fill opacity=0.2] coordinates {
	(-0.666906517,2.000252053)
	(-0.666906517,0.85570584)
	(0.180945996,0.85570584)
	(0.180945996,2.000252053)
	(-0.666906517,2.000252053)
};
\end{axis}
\end{scope}
\node at (4.5,1.4) {$
	\begin{array}{c}
	\begin{pmatrix}p_1+p_2 & 1\\[4pt] 0 & 2p_2\end{pmatrix}\begin{pmatrix}
	x_1 \\[4pt] x_2
	\end{pmatrix}=\begin{pmatrix}
	1 \\[4pt] 2
	\end{pmatrix}\\[18pt]
	\begin{array}{l}
	p_1\in[1,2] \\[4pt]
	p_2\in[0.5,1]
	\end{array}
	\end{array}$};
\end{tikzpicture}
\caption{Comparison of the $p$-solution (dark gray region) and the interval solution (light gray region) for a~given two dimensional IPLS with two parameters}
\label{fig:psol_vs_intsol}
\end{figure}


\subsection{Intervals and affine forms}

A~real compact interval is $\inum{x}=[\underline x,\overline x]=\{x\in\mathbb{R}\mid\underline x\leqslant x\leqslant\overline x\}$, where $\underline x,\overline x\in\reals$. The mid-point $x^c=(\underline x+\overline x)/2$, radius $x^{\Delta}=(\overline x-\underline x)/2$, and the maximal absolute value (\textit{magnitude}) $|\inum{x}|=\max\{|x|\;|\;x\in\inum{x}\}$ are applied to interval vectors and matrices componentwise. By $\ints^n$ and $\ints^{n\times m}$ we denote the set of all $n$-dimensional interval vectors and the set of all $n\times m$ interval matrices, respectively. The identity matrix of size $n$ is denoted by~$I_n$, and for a~non-empty bounded set $S\subset\reals^n$, its \textit{interval hull} is defined as
\[
\ihull{S}=\bigcap\{\inum{y}\in\ints^n\mid S\subseteq\inum{y}\}.
\]
The $i$-th column of a~matrix $A$ is denoted by $A_{*i}$ and $A=(A_{*1}\,\ldots\,A_{*n})$ is a~column-wise notation of $A$. The spectral radius of a~real matrix $A$ is denoted by $\rho(A)$. 

\paragraph{Revised affine forms (RAF).} A~revised affine form (cf.~\cite{Skalna:2018:PIAS,SkaHla:2017:CMAAF,VuHaroud:2004:GSCM}) of length $n$ is defined as a~sum of a~standard affine form (see, e.g., \cite{Comba:1993:AAA}) and a~term that represents all errors introduced during a~computation (including rounding errors), i.e.,
\begin{equation}
\label{eq:reducedaaf}
\hat x=x_0+e^Tx+x_r[-1,1],
\end{equation}
where $e=(\varepsilon_1,\ldots,\varepsilon_n)^T$ and $x=(x_1,\ldots,x_n)^T$. The {\em noise symbols} $\varepsilon_i$, $i=1,\ldots,n$, are unknown, but assumed to vary independently within the interval $[-1,\,1]$, and $x_r\geqslant 0$ is the radius of the \textit{accumulative error} $x_r[-1,1]$. 
If two revised affine forms $\hat x$ and $\hat y$ share the same noise symbols $\varepsilon_i$, this means that there is a~partial dependency between them, and partial deviations $x_i$ and $y_i$ determine the magnitude and sign of this dependency. The length of a~revised affine form equals to the number of initial uncertain parameters and remains unchanged during the same computation, unless new independent parameters are introduced. Let us notice that a~revised affine form is an~interval-affine function of the noise symbols, so it can be written as $\inum{x}(e)=e^Tx+\inum{x}$, where $\inum{x}=x_0+x_r[-1,1]$. 

In order to perform computation on RAFs, we must define elementary operations for them. 
It is not hard to see that affine-linear operations result straightforwardly in a~revised affine form. However, the result of a~nonlinear operation must be approximated by a~revised affine form, and the error of this approximation must be taken into account. An overview of operations on revised affine forms can be found, e.g., in \cite{Skalna:2018:PIAS,SkaHla:2017:CMAAF}. Below, we recall the Chebyshev minimum-error multiplication of revised affine forms (cf.~\cite{SkaHla:2017:CMAAF}), which is especially important for the quality of the enclosures produced by interval-affine methods. Thus, given two revised affine forms $\hat x=e^Tx+x_r[-1,1]$ and $\hat y=e^Ty+y_r[-1,1]$, the revised affine form $\hat z$ that approximates $\hat x\cdot\hat y$ is defined by the following formula:
\begin{equation}
\hat z=x_0y_0+d^c+e^T(x_0y_i+y_0x_i)+\left(|x_0|y_r+|y_0|x_r+d^{\Delta}\right)[-1,1],
\end{equation}
where $[d^c-d^{\Delta},d^c+d^{\Delta}]$ is the range of $f(x,y)=xy$ on the joint range (cf.~\cite{Comba:1993:AAA}) $\langle e^Tx+x_r\varepsilon_x,e^Ty+y_r\varepsilon_y\rangle$, which is the set of all possible values of
$(\hat x,\hat y)$ when the noise symbols $\varepsilon_1,\ldots,\varepsilon_n$ are independently chosen in the interval $[-1,1]$.

\paragraph{Affine transformation.} An interval parameter $\inum{p}_k$ can be represented by the revised affine form $\hat p_k=\Mid{p}_k+\Rad{p}_k\varepsilon_k$. If, for $k=1,\ldots,K$, we substitute $\hat p_k$ for $\inum{p}_k$ and then perform respective operations on revised affine forms $\hat p_k$, we obtain the following interval-affine linear system:
\begin{equation}
\label{eq:affintlinsys}
\bnum{C}(e)x=\bnum{c}(e),
\end{equation}
where
\begin{subequations}
	\begin{alignat}{4}
	\label{eq:pils_a_b_aff}
	\bnum{C}(e) &=C^{(0)}+\textstyle\sum_{k=1}^{K}C^{(k)}\varepsilon_k+C^r[-1,1]=
	\textstyle\sum_{k=1}^{K}C^{(k)}\varepsilon_k+\bnum{C}, \\
	\bnum{c}(e)&=c^{(0)}+\textstyle\sum_{k=1}^{K}c^{(k)}\varepsilon_k+c^r[-1,1]=
	\textstyle\sum_{k=1}^{K}c^{(k)}\varepsilon_k+\bnum{c}.
	\end{alignat}
\end{subequations}
We call the transition from the system \nref{eq:pils} to the system \nref{eq:affintlinsys} the \textit{affine transformation}.

\paragraph{Remark.} Whenever the vector of parameters is denoted by $e$, this means that its components vary within the interval $[-1,1]$. Otherwise, i.e., if the vector of parameters is denoted by $p$, this means that we do not make any specific assumptions about the range of variability of the components of $p$.

Recall the observation that was shown, e.g., in Skalna \& Hlad{\'i}k \cite{SkaHla2017b}.

\begin{proposition}
	\label{prop:paramtoaff}
	The solution set of the system (\ref{eq:pils}) is included in the solution set of the system (\ref{eq:affintlinsys}).
\end{proposition}
In view of the above, we can restrict our considerations to interval-affine linear systems, i.e., systems with entries being revised affine forms. In fact, we can restrict the considerations to interval parametric linear systems with affine-linear dependencies only, since each accumulative error can be treated as an~independent noise symbol.

\subsection{Preconditioning}

In order that the parametric solution set $\Solset$ be bounded, the matrix $A(\inum{p})$ must be regular, i.e., $A(p)$ must be nonsingular for each $p\in\inum{p}$. If $A(p)$ is singular for some $p\in\inum{p}$, then $A(\inum{p})$ is singular and the solution set is either empty (which is rare) or unbounded.

So, regularity of $A(\inum{p})$ is sufficient for the boundedness of a~parametric solution set. However, most of the methods for solving IPLS require that $A(\inum{p})$ is so called \emph{strongly regular} (cf. \cite{Popova:2004:SRPIM,Pop2018a,Skalna:2017:SRPM,Skalna:2018:PIAS}). Notice that there is ambiguity in using the notion of strong regularity for parametric matrices; we adopt here the following definition.

\begin{definition}
An interval parametric matrix $A(\inum{p})$ is strongly regular if the midpoint matrix $A(\Mid{p})$ is nonsingular and at least one of the interval hulls $\inum{H}$, $\inum{H}'$ of the matrices
\begin{equation}
	\label{eq:b_and_bprime}
	H(\inum{p})=\left\{A(\Mid{p})^{-1}A(p)\mid p\in\inum{p}\right\}, \quad
	H'(\inum{p})=\left\{A(p)A(\Mid{p})^{-1}\mid p\in\inum{p}\right\},
\end{equation}
are regular.
\end{definition}

Let us notice that if there are only affine-linear dependencies in \nref{eq:pils}, then
\begin{equation}
\label{eq:b_and_bprime_lincase}
\inum{H}=I_n+H^{\Delta}[-1,1], \quad
\inum{H}'=I_n+(H')^{\Delta}[-1,1],
\end{equation}
where $H^{\Delta}=\sum_{k=1}^K\left|A(p^c)^{-1}A^{(k)}\right|p^{\Delta}_k$, $(H')^{\Delta}=\sum_{k=1}^K\left|A^{(k)}A(p^c)^{-1}\right|p^{\Delta}_k$.

The concept of strong regularity is strictly related to \textit{preconditioning}. As it is well known, preconditioning aims to make a~system more suitable for numerical (especially iterative) methods. Left preconditioning, which corresponds to linear transformation of the right-hand side, is used the most often. Indeed, the majority of methods for solving IPLS left-precondition the system by the midpoint inverse $A(\Mid{p})^{-1}$ and require/check regularity of the interval matrix $\inum{H}$. However, strong regularity of $A(\inum{p})$ can be ascertained via regularity of $\inum{H}'$ despite $\inum{H}$ being not regular (cf.~\cite{Popova:2004:SRPIM}).

In this paper, we propose a~novel\footnote{To our best knowledge, both right and double preconditioning were not yet considered in the context of solving interval parametric linear systems.} approach to solving interval parametric linear systems which employs right and double preconditioning. The main advantages of the proposed approach are the following:
\begin{itemize}
\item[$-$] it handles interval parametric linear systems with both linear and nonlinear dependencies,
\item[$-$] it produces a~$p$-solution represented by a~revised affine form, which preserves first order dependencies between the solution and input parameters,
\item[$-$] it enables to solve a~wider class of problems than most known methods for solving IPLS, and
\item[$-$] it enables to handle larger uncertainties. 
\end{itemize}

\section{Right preconditioning}
\label{sec:main}
In order to make an~IPLS more tractable, it is usually left preconditioned with the inverse of the midpoint matrix (midpoint inverse), $R=(\Mid{A})^{-1}$, or rather its numerical approximation (cf.~\cite{Hladik:2016:OPFIM,Skalna:2017:SRPM}). Left preconditioning of the system \nref{eq:pils} yields the new system
\begin{equation}
\label{eq:affintlinsys2_precond}
H(\inum{p})x=d(\inum{p}),
\end{equation}
where $H(\inum{p})=RA(\inum{p})$, $d(\inum{p})=Rb(\inum{p})$.
The following was shown in many papers, see, e.g., Skalna \cite{Skalna:2018:PIAS}.

\begin{proposition}
The solution set of the system (\ref{eq:affintlinsys}) is included in the solution set of the system (\ref{eq:affintlinsys2_precond}).
\end{proposition} 

Analogously, right preconditioning of the system \nref{eq:affintlinsys} by matrix $R$ yields the system
\begin{equation}
\label{eq:affintlinsys2_postcond}
H'(\inum{p})y=b(\inum{p}),
\end{equation}
where $H'(\inum{p})=A(\inum{p})R$ and $y$ is such that $Ry=x$.
Right preconditioning for standard interval linear systems of equations was investigated in \cite{Gol2005,Neu1987}, but for parametric systems it seems that the problem was not analyzed in detail yet (besides the problem of regularity \cite{Popova:2004:SRPIM}).

\begin{proposition}
The solution set of the system (\ref{eq:affintlinsys}) is included in the product of $R$ and the solution set of the system (\ref{eq:affintlinsys2_postcond}).
Consequently, $\inum{x}(p):=R\cdot \inum{y}(p)$ is a~$p$-solution of (\ref{eq:affintlinsys}), where $\inum{y}(p)$ is a~$p$-solution of (\ref{eq:affintlinsys2_postcond}).
\end{proposition}

Most of the methods for solving IPLS perform left preconditioning and require that the matrix $\inum{H}$ is regular, which assures strong regularity of $A(\inum{p})$. However, $A(\inum{p})$ may be ascertained via the regularity of $\inum{H}'$ despite $\inum{H}$ is not regular. This property was already observed, e.g., in Popova~\cite{Pop2018a}.

Left preconditioning is most natural since no extra step is needed to obtain the final solution. When solving interval parametric linear systems, this last step usually causes that the produced enclosures are pretty rough. Using the approach proposed in this paper we can significantly reduce this overestimation.

\paragraph{Properties of right preconditioning.}
Popova \cite{Popova:2014:IESP,PopHla2013} defined a~parameter $p_k$ to be of class one if nonzero elements in $(A^{(k)}\mid b^{(k)})$ are in at most one row, that is, the parameter $p_k$ affects one equation only. If all parameters are of class one, many problems become much easier. For example, we can explicitly describe the solutions of the system \nref{eq:parintlinsys} by
\begin{align*}
|A(\Mid{p})x-b(\Mid{p})|\leqslant\sum_{k=1}^K\Rad{p}_k\left|A^{(k)}x-b^{(k)}\right|.
\end{align*}

\begin{proposition}
If a~parameter $p_k$ is of class one, then it remains to be of class one after right preconditioning.
\end{proposition}

\begin{proof}
From the assumptions the matrix $A^{(k)}$ and vector $b^{(k)}$ have the form of $A^{(k)}=e_iv^T$ and $b^{(k)}=e_id$, for some $v\in\R^n$ and $d\in\R$. Right preconditioning with $R$ transforms $A^{(k)}$ into $(e_iv^T)R=e_i(v^TR)$, and $b^{(k)}$ does not change. Thus, the parameter $p_k$ remains to be of class one.
\end{proof}

\paragraph{How to perform left/right preconditioning?}
If we first precondition, and then relax dependencies, we obtain the interval matrix
$$\textstyle
\inum{H}_1:=A^{(0)}+\sum_{k=1}^K (RA^{(k)})\inum{p}_k,
$$
whereas the converse order yields
$$\textstyle
\inum{H}_2:=A^{(0)}+R\left(\sum_{k=1}^K  A^{(k)}\inum{p}_k\right).
$$
Due to subdistributivity of interval arithmetic, $\inum{H}_1\subseteq\inum{H}_2$ (cf.\ \cite{Hladik:2012:EFS}), so the first approach gives tighter or the same enclosures. Equality $\inum{H}_1=\inum{H}_2$ appears for standard interval matrices, and we will extend this class for certain interval parametric matrices as follows.

\begin{proposition}\label{propColOnePrecond}
Suppose that for every $k=1,\dots, K$ the matrix $A^{(k)}$ has at most one non-zero element in each column. Then $\inum{H}_1=\inum{H}_2$.
\end{proposition}

\emph{Remark.} For right preconditioning the result is similar; matrices $A^{(k)}$ just have to possess at most one non-zero element in each row.

\begin{proof}
Let $i,j$ be fixed indices and for each $k=1,\dots,K$ define $\alpha_{jk}$ to be that index for which $A^{(k)}_{\alpha_{jk},j}\not=0$ (if there is no such an index, then choose an arbitrary one). Now,
\begin{align*}
\left(\inum{H}_1-A^{(0)}\right)_{ij}
&\textstyle
=\left(\sum_{k=1}^K (RA^{(k)})\inum{p}_k\right)_{ij}
=\sum_{k=1}^K \left(\sum_{\alpha=1}^n
   R_{i\alpha}A^{(k)}_{\alpha j}\right)\inum{p}_k \\
&\textstyle=\sum_{k=1}^K R_{i\alpha_{jk}} A^{(k)}_{\alpha_{jk} j}\inum{p}_k,
\end{align*}
and
\begin{align*}
\left(\inum{H}_2-A^{(0)}\right)_{ij}
&\textstyle
=\left(R\sum_{k=1}^K  A^{(k)}\inum{p}_k\right)_{ij}
=\sum_{\alpha=1}^n R_{i\alpha}
   \left(\sum_{k=1}^K A^{(k)}_{\alpha j}\inum{p}_k\right)_{ij}\\\textstyle
&=\sum_{\alpha=1}^n R_{i\alpha}
   \sum_{k:\alpha=\alpha_{jk}} A^{(k)}_{\alpha j}\inum{p}_k
=\sum_{k=1}^K R_{i\alpha_{jk}} A^{(k)}_{\alpha_{jk} j}\inum{p}_k,
\end{align*}
which concludes the proof.
\end{proof}

\begin{remark}[Order of preconditioning and relaxation]
The above result shows that for those types of parametric matrices it is better to first relax dependencies and then to precondition. The saving of time complexity is significant. The computation of $\inum{H}_1$ costs $\mna{O}(Kn^3)$, whereas the  computation of $\inum{H}_2$ costs only $\mna{O}(Kn^2+n^3)$. The assumptions of Proposition~\ref{propColOnePrecond} satisfies a~nontrivial class of parametric matrices, including symmetric, skew-symmetric or Toeplitz and other special interval parametric matrices \cite{AleKre2003,May2017,Hla2008g}.
\end{remark}

\section{Double preconditioning}\label{sec:double_prec}
As shown by Neumaier~\cite{Neumaier:1990:IMS}, left preconditioning is sufficient for checking (strong) regularity of standard interval matrices. Right preconditioning is equivalent with respect to strong regularity. So, using simultaneously both of them is useless. Surprisingly, for interval parametric matrices, the converse is true! The following example illustrates that there are interval parametric matrices for which neither left nor right preconditioning helps, but a~suitable combination of both works.

\begin{example}\normalfont
	\label{ex:ex0}
Consider the interval parametric matrix
\begin{align*}
A(p)=\begin{pmatrix}1-0.5p & -p\\ 0.5p & 1+p\end{pmatrix},\quad
p\in[-1,1].
\end{align*}
Since its midpoint is the identity matrix, i.e., $A(\Mid{p})=I_2$, both left and right preconditioning yield the same interval matrix
$$
\inum{H}=\inum{H}'=
\begin{pmatrix}[0.5,\,1.5] & [-1,1]\\{} [-0.5,\,0.5] & [0,2]\end{pmatrix},
$$
which is not regular as $\rho(\Rad{H})=1.5$. Nevertheless, taking
\begin{align*}
R=\begin{pmatrix}1 & 1\\ 0 & 1\end{pmatrix},\quad
R^{-1}=\begin{pmatrix}1 & -1\\ 0 & 1\end{pmatrix},
\end{align*}
and preconditioning the parametric matrix both from left and from right, we obtain
\begin{align*}
\inum{H}''
=\hull\braces*{RA(p)R^{-1}\mid p\in\inum{p}}
=\begin{pmatrix}1 & 0\\{} [-0.5,\,0.5] & [0.5,1.5]\end{pmatrix},
\end{align*}
which is regular since $\rho\big(({H}'')\Rad{}\big)=0.5$.
\end{example}

This example shows that the concept of strong regularity for parametric matrices should be extended to include double preconditioning. 
\begin{definition}
	A~parametric matrix $A(\inum{p})$ is strongly regular if $A(\Mid{p})$ is nonsingular and the interval hull $\inum{H}$ of the matrix
	\begin{equation}\label{eq:strreg_dbl}
	H(\inum{p})=\left\{LA(p)R\mid p\in\inum{p}\right\}
	\end{equation}
	is regular for some $R$ and $L$ such that $A(\Mid{p})^{-1}=RL$.
\end{definition}
If there are only affine-linear dependencies in \nref{eq:strreg_dbl} then
\begin{equation}\label{eq:b_and_bprime2_lincase}
\inum{H}=I+H^{\Delta}[-1,1]
\end{equation}
where $H^{\Delta}=\sum_{k=1}^K\left|LA^{(k)}R\right|p^{\Delta}$.
The case $L=A(\Mid{p})^{-1}$, $R=I_n$ corresponds to left preconditioning and the case $L=I_n$, $R=A(\Mid{p})^{-1}$ to right preconditioning. However, it is not clear what is the best choice for $L,R$, and it seems to be a~challenging problem. Thus, as a~first step, we consider some simple cases. Even though these simple cases can be solved directly, the idea is to find a~general approach that will work very well even for these cases.

\paragraph{Rank one matrix.}
Consider the parametric matrix in the specific form with one parameter and rank one matrix
\begin{align*}
A(p)=I_n+A^{(1)}p_1,
\end{align*}
where $p_1\in\inum{p}_1=[-1,1]$. The absolute term is the identity matrix, which can be obtained by standard preconditioning by the midpoint inverse. Assume that $A^{(1)}$ has rank one, so it can be written as $A^{(1)}=ab^T$ for some $a,b\in\R^n$. This very special form of $A(p)$ can be handled analytically, but we put it into the standard framework to possibly come up with a~heuristic for the~general case.

The commonly used sufficient condition for testing regularity is $\rho(|ab^T|)<1$, which takes the form $\rho(|ab^T|)=\rho(|a|\cdot|b|^T)=|b|^T|a|<1$. Double preconditioning yields
$$
RA(p)R^{-1}=I_n+(Rab^TR^{-1})p_1,
$$
so the regularity test reads
$$
\rho(|Rab^TR^{-1}|)
=\rho(|Ra|\cdot|b^TR^{-1}|)
=|b^TR^{-1}|\cdot|Ra|<1.
$$
For which $R$ is the left-hand side minimal? First, we derive its lower bound
$$
|b^TR^{-1}|\cdot|Ra|
\geqslant |b^TR^{-1}Ra| =|b^Ta|.
$$
Now, we show that this lower bound is attained for example for $R$ such that $Ra=e_1$:
$$
\rho(|Rab^TR^{-1}|)
=|b^TR^{-1}|\cdot|Ra|
=|b^TR^{-1}Ra|
=|b^Ta|.
$$

\paragraph{Several rank one matrices.}
Can we extend the above idea to more parameters? 
Consider the parametric matrix 
\begin{align*}
A(p)=I_n+\sum_{k=1}^K A^{(k)}p_k,
\end{align*}
where, for each $k=1,\dots,K$, $A^{(k)}=a^{(k)}(b^{(k)})^T$ has rank one. Provided vectors $a^{(1)},\dots,a^{(K)}$ are linearly independent, we can easily find $R$ such that $Ra^{(k)}$ is a~canonical unit vector for each~$k$. The bad news is that the spectral radius is not additive, so this choice of $R$ needn't be optimal. In the following example we test how good is the choice proposed above.

\begin{example}\normalfont
	\label{ex:romatrices_1}
	 Consider random $n\times n$ matrices that are obtained as follows: we draw the elements of the vectors $a^{(k)}$ and $b^{(k)}$, $k=1,\dots,K$, from the intervals $\inum{u}+[-0.3k,0.3k]$ and $\inum{v}+[-0.3k,0.3k]$, respectively, where $\inum{u}=[-0.5,1.0]$ and $\inum{v}=[2.0,2.5]$. Then $A(p)=I_n+\sum_{k=1}^Ka^{(k)}(b^{(k)})^Tp_k$ in the first variant, and $A(p)=I_n+\sum_{k=1}^Ka^{(k)}(a^{(k)})^Tp_k$ in the second variant. The parameters $p_k$, $k=1,\ldots,K$, are assumed to vary within the interval $[-1,1]$. The preconditioning matrix $R$ described above is obtained as follows: if $K=n$, then $R=A^{-1}$, where $A=(a^{(1)},\dots,a^{(n)})$; if $K<n$, we extend $A=(a^{(1)},\dots,a^{(K)})$ to a~square matrix $A'$ by adding some artificial~vectors (linearly independent of $a^{(k)}$ and of each other) and take $R=(A')^{-1}$. This strategy of computing $R$ will be referred to as S0 strategy. In order to find a~possibly better preconditioning matrix we consider the following three other strategies: 
\newcounter{strat}
\setcounter{strat}{1}
\addtolength{\leftmargini}{2ex}
\begin{itemize}
\item[S\thestrat:] Double preconditioning with $R$ and $R^{-1}$, where $R$ is obtained from the spectral decomposition of $A^{(k)}$ having the highest norm.
\stepcounter{strat}
\item[S\thestrat:] Double preconditioning with $R$ and $R^{-1}$, where $R$ is obtained from the spectral decomposition of $A(p)$, where $p\in\inum{p}$ is a~random vector of parameters; we take $R$ which gives the minimum value of $\rho(\Rad{H})$  out of 1000 repetitions.
\stepcounter{strat}
\item[S\thestrat:] Double preconditioning with $R$ and $R^{-1}$, where $R$ is obtained from the spectral decomposition of $A(p)$, where $p\in\inum{p}$ is a~random combination of the endpoints of $\inum{p}$; we take $R$ which gives the minimum value of $\rho(\Rad{H})$  out of 1000 repetitions.
\end{itemize}
Denote $A^{\Delta}:=\sum_{k=1}^K\left|A^{(k)}\right|p^{\Delta}_k$ and $\Rad{H}:=\sum_{k=1}^K\left|RA(p)R^{-1}\right|p^{\Delta}_k$. Fig.~\ref{fig:compar_spradii_DPex1} shows the box plot of the ratio of the spectral radii $\rho(\Rad{A})/\rho(\Rad{H})$ for the first variant.
\begin{figure}
\centering
\includegraphics[width=0.9\textwidth,trim={15pt 15pt 15pt 11pt},clip]{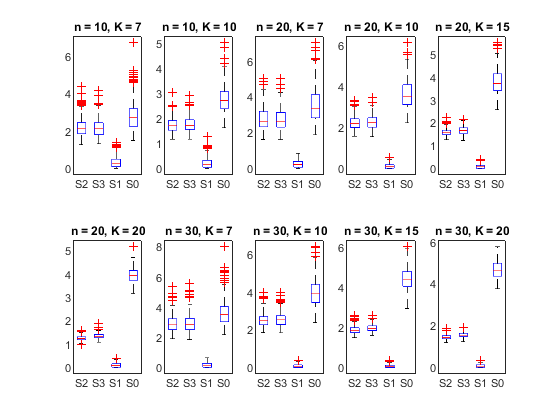}
\caption{Results for Example~\ref{ex:romatrices_1}: boxplot of ratios $\rho(\Rad{A})/\rho(\Rad{H})$ obtained from 500 repetitions for first variant; $\inum{u}=[-0.5,1.0]$, $\inum{v}=[2,2.5]$}
	\label{fig:compar_spradii_DPex1}
\end{figure}
As can be seen from the figure, the S0 strategy is significantly better than other considered strategies; the advantage grows with $n$ and $K$ (notice that the influence of $K$ is even greater). To better illustrate the differences between the considered strategies, we give the geometric means of the ratios $\rho(\Rad{A})/\rho(\Rad{H})$ in Table~\ref{tab:compar_spradii_DPex1}. The results in the table show that the S0 strategy is better also for the second variant.

\begin{table}
\centering
\caption{Results for Example~\ref{ex:romatrices_1}: the geometric means of ratios $\rho(\Rad{A})/\rho(\Rad{H})$ for four considered preconditioning strategies and two variants of $A(p)$; the best results are typed in boldface}
	\begin{tabular}{rr|rrrr|rrrr}
		\hline
		\multicolumn{2}{c|}{Sizes} & \multicolumn{4}{c|}{First variant} &
		\multicolumn{4}{c}{Second variant} \\
		\hline
		\multicolumn{1}{l}{$n$} & \multicolumn{1}{l|}{$K$} & \multicolumn{1}{c}{ S2} & \multicolumn{1}{c}{ S3} & \multicolumn{1}{c}{ S1} & \multicolumn{1}{c|}{ S0} & \multicolumn{1}{c}{ S2} & \multicolumn{1}{c}{ S3} & \multicolumn{1}{c}{ S1} & \multicolumn{1}{c}{ S0} \\
		\hline
		10    & 7     & 2.10  & 2.09  & 0.21  & \textbf{2.63} & 1.65  & 1.54  & 1.30  & \textbf{1.86} \\
		10    & 10    & 1.71  & 1.72  & 0.08  & \textbf{2.72} & 1.65  & 1.55  & 1.30  & \textbf{2.09} \\
		20    & 7     & 2.65  & 2.64  & 0.20  & \textbf{3.27} & 1.98  & 1.77  & 1.34  & \textbf{2.19} \\
		20    & 10    & 2.14  & 2.19  & 0.05  & \textbf{3.38} & 1.98  & 1.79  & 1.32  & \textbf{2.45} \\
		20    & 15    & 1.64  & 1.71  & 0.05  & \textbf{3.76} & 1.91  & 1.73  & 1.30  & \textbf{2.81} \\
		20    & 20    & 1.25  & 1.33  & 0.05  & \textbf{3.80} & 1.77  & 1.64  & 1.27  & \textbf{3.03} \\
		30    & 7     & 2.99  & 2.96  & 0.10  & \textbf{3.61} & 2.13  & 1.89  & 1.35  & \textbf{2.32} \\
		30    & 10    & 2.50  & 2.53  & 0.07  & \textbf{3.95} & 2.23  & 1.97  & 1.33  & \textbf{2.67} \\
		30    & 15    & 1.91  & 1.98  & 0.05  & \textbf{4.47} & 2.19  & 1.91  & 1.29  & \textbf{3.15} \\
		30    & 20    & 1.48  & 1.60  & 0.04  & \textbf{4.76} & 2.04  & 1.80  & 1.28  & \textbf{3.40} \\
		\hline
	\end{tabular}%
	\label{tab:compar_spradii_DPex1}%
\end{table}%
	
\end{example}
%

\paragraph{Higher rank matrices.}
Now consider a~parametric matrix with one parameter
\begin{align*}
A(p)=I_n+A^{(1)}p_1,
\end{align*}
where $A^{(1)}=AB^T$ has rank $r$ and $A,B\in\R^{n\times r}$. Extending the previous ideas, we can think of a~matrix $R$ such that $RA=\binom{I_r}{0}$. This matrix is easy to find, and the regularity condition then reads
\begin{align*}
\rho(|RAB^TR^{-1}|)
&=\rho(|RA|\cdot|B^TR^{-1}|)
=\rho(|B^TR^{-1}|\cdot|RA|)\\
&=\rho(|B^TR^{-1}RA|)
=\rho(|B^TA|).
\end{align*}
Even though this choice of $R$ needn't be optimal for the regularity condition, it seems to be a~promising candidate. The reason is that $\rho(|B^TA|)$ is supposed to be often smaller (since the matrix is smaller) than the value $\rho(|AB^T|)$ resulting from the case when no preconditioning is used.

\paragraph{Yet higher rank matrices.}
We again consider a~parametric matrix with one parameter
\begin{align*}
A(p)=I_n+A^{(1)}p_1,
\end{align*}
and the matrix $A^{(1)}$ has full rank or almost full rank. Matrix $A^{(1)}$ is similar to some simpler form matrix $D$ (e.g., diagonal, Jordan form, or other) by a~similarity transformation $RA^{(1)}R^{-1}=D$. So we consider double preconditioning by $R$ and $R^{-1}$, yielding $RA(p)R^{-1}=I_n+RA^{(1)}R^{-1}p_1=I_n+Dp_1$. 

When $D$ is real diagonal, then $\rho(D)=\rho(|D|)$ and we have the best bound. More realistically,  $A^{(1)}$ has some complex eigenvalues. Then we can consider $D$ to be real block diagonal: the blocks of size one correspond to real eigenvalues and  the blocks of size two correspond to the pair of complex conjugate eigenvalues. Then $\rho(D)\leqslant\rho(|D|)\leqslant \sqrt{2}\rho(D)$ since the worst case is the block of the form $\left(\begin{smallmatrix}1&1\\-1&1\end{smallmatrix}\right)$; see Proposition~\ref{propBlockRotRho} below. Even more, for this matrix, no double preconditioning of the particular blocks improves the bound (Proposition~\ref{propBlockRotPrec}).

\begin{proposition}\label{propBlockRotRho}
For each block $B$ of matrix $D$ we have $\rho(|B|)\leqslant \sqrt{2}\rho(B)$ and the bound is tight for $B=\left(\begin{smallmatrix}1&1\\-1&1\end{smallmatrix}\right)$.
\end{proposition}

\begin{proof}
Block $B$ has the form of $B=\left(\begin{smallmatrix}c&s\\-s&c\end{smallmatrix}\right)$. Its eigenvalues are $c\pm si$ and its spectral radius $\rho(B)=\sqrt{c^2+s^2}$. For the absolute value of $B$ we have $\rho(|B|)=c+s$. We want to know the minimum of 
$$
f(c,s)=\frac{\rho(B)^2}{\rho(|B|)^2}=\frac{c^2+s^2}{(c+s)^2}.
$$
Since it is invariant to scaling, we normalize it such that $c=1$ (case $c=0$ is trivial). The derivative of function  
$
f(s)=\frac{1+s^2}{(1+s)^2}
$
is 
$f'(s)=\frac{2s-2}{(1+s)^3}$, which is zero for $s=1$. Indeed, this corresponds to the minimum of $f(s)$, so the best bound is attained for $B=\left(\begin{smallmatrix}1&1\\-1&1\end{smallmatrix}\right)$.
\end{proof}

\begin{proposition}\label{propBlockRotPrec}
For each block $B$ of matrix $D$ and each nonsingular preconditioner $R$ we have $\rho(|B|)\leqslant \rho(|RBR^{-1}|)$.
\end{proposition}

\begin{proof}
Without loss of generality assume that block $B$ has the form of $B=\left(\begin{smallmatrix}f&1\\-1&f\end{smallmatrix}\right)$, where $f\geqslant 0$. Consider the preconditioner in the form $R=\left(\begin{smallmatrix}a&b\\c&d\end{smallmatrix}\right)$ normalized such that $ad-bc=1$, that is, its determinant is~1. Then
\begin{align*}
RBR^{-1}
&=\begin{pmatrix}a&b\\c&d\end{pmatrix}
  \begin{pmatrix}f&1\\-1&f\end{pmatrix}
  \begin{pmatrix}d&-b\\-c&a\end{pmatrix}\\
&=\begin{pmatrix}f(ad-bc)-(bd+ac)&a^2+b^2\\-c^2-d^2&f(ad-bc)+(bd+ac)\end{pmatrix}.
\end{align*}
Denote
$$
C:=
\begin{pmatrix}f(ad-bc)-(bd+ac)&a^2+b^2\\c^2+d^2&f(ad-bc)+(bd+ac)\end{pmatrix}.
$$
Then $|C|\leqslant|RBR^{-1}|$, so $\rho(C)\leqslant\rho(|RBR^{-1}|)$. Denote by $\lambda_1\geqslant\lambda_2$ the eigenvalues of~$C$. They are real since the off-diagonal entries of $C$ are nonnegative. The trace of $C$ is $\lambda_1+\lambda_2=2f(ad-bc)=2f$.

Since $\rho(|B|)=f+1$, we suppose to the contrary that $\rho(|RBR^{-1}|)<f+1$. Thus also $\rho(C)<f+1$. Since $\lambda_1<f+1$, we get $\lambda_2>f-1$. Hence $C-(f-1)I_2$ has positive eigenvalues and its determinant must be positive, too. Thus
\begin{align*}
0
&<\det(C-(f-1)I_2)\\
&=\det\begin{pmatrix}(ad-bc)-(bd+ac)&a^2+b^2\\c^2+d^2&(ad-bc)+(bd+ac)\end{pmatrix}\\
&=(ad-bc)^2-(bd+ac)^2 -(a^2+b^2)(c^2+d^2)\\
&=a^2d^2+b^2c^2-2abcd-b^2d^2-a^2c^2-2abcd -a^2c^2-b^2c^2-a^2d^2-b^2d^2\\
&=-2b^2d^2-2a^2c^2-4abcd
=-2(ac+bd)^2
\leqslant 0,
\end{align*}
which is a~contradiction.
\end{proof}

\paragraph{Several higher rank matrices.}
In the general case,
\begin{align*}
A(p)=I_n+\sum_{k=1}^K A^{(k)}p_k.
\end{align*}
The above discussion motivates us to construct the preconditioners such that it diagonalizes the matrix $A^{(k)}$ with highest norm, which corresponds to the S2 strategy. Unfortunately, the next example shows that in the general case the S2 strategy is not as efficient as it was expected. However, it might be useful for symmetric matrices.

\newcounter{strathr}
\setcounter{strathr}{1}

\begin{example}\normalfont
	\label{ex:hrmatrices}
	Consider an interval parametric matrix with $A^{(k)}$, for $k=1,\ldots,K$, of rank $r>1$. Each $A^{(k)}$ is obtained as follows (the procedure is similar to the one used in Example \ref{ex:romatrices_1}): first the elements of full rank matrices $A_k,B_k\in\R^{n\times r}$ are drawn from the intervals $\inum{u}+[-0.3k,0.3k]$ and $\inum{v}+[-0.3k,0.3k]$, respectively, where  $\inum{u}=[-0.5,1.0]$, $\inum{v}=[2.0,2.5]$. Then we put $A(p)=I_n+\sum_{k=1}^KA_kB_k^T$ in the first variant, and $A(p)=I_n+\sum_{k=1}^KA_kA_k^T$ in the second variant.
	The results obtained by using preconditioning strategies S1, S2, S3 for the first variant are presented in Fig.~\ref{fig:compar_spradii_DPex2HR}. Additionally, Table \ref{tab:compar_spradii_DPex2HR} shows the geometric means of the ratios $\rho(\Rad{A})/\rho(\Rad{H})$ for two considered cases. As we can see from the results, the S3 strategy is the best for the first variant, whereas for the second variant, the S2 strategy is slightly better than two other strategies.
	
\begin{figure}
	\centering
\includegraphics[width=0.9\textwidth,trim={15pt 15pt 15pt 11pt},clip]{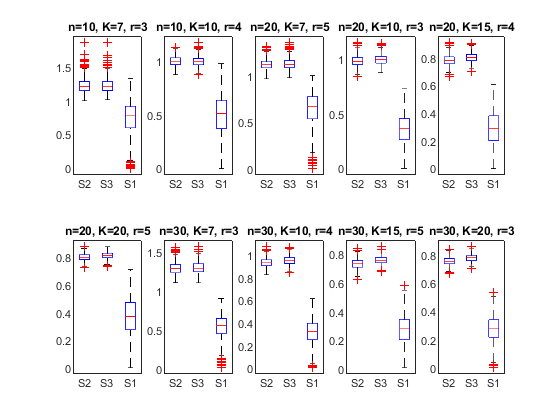}
\caption{Results for Example~\ref{ex:hrmatrices}: boxplot of ratio $\rho(\Rad{A})/\rho(\Rad{H})$ obtained from 500 repetitions for first variant; $\inum{u}=[-0.5,1.0]$, $\inum{v}=[2,2.5]$}
	\label{fig:compar_spradii_DPex2HR}
\end{figure}

\begin{table}
\centering
\caption{Results for Example~\ref{ex:hrmatrices}: the geometric means of ratio $\rho(\Rad{A})/\rho(\Rad{H})$ for considered preconditioning strategies; the best results are typed in boldface}
\begin{tabular}{rrr|rrr|rrr}
\hline
\multicolumn{3}{c|}{Sizes} & \multicolumn{3}{c|}{First variant} & \multicolumn{3}{c}{Second variant} \\
\hline
\multicolumn{1}{l}{$n$} & \multicolumn{1}{l}{$K$} & \multicolumn{1}{l|}{ rank} & \multicolumn{1}{c}{ S2} & \multicolumn{1}{c}{ S3} & \multicolumn{1}{c|}{ S1} & \multicolumn{1}{c}{ S2} & \multicolumn{1}{c}{ S3} & \multicolumn{1}{c}{ S1} \\
\hline
10    & 7     & 3     & \textbf{1.24} & \textbf{1.24} & 0.67  & \textbf{1.22} & 1.19  & 1.16 \\
10    & 10    & 4     & \textbf{1.01} & \textbf{1.01} & 0.45  & \textbf{1.06} & \textbf{1.06} & 1.05 \\
20    & 7     & 5     & \textbf{1.13} & \textbf{1.13} & 0.61  & \textbf{1.20} & 1.16  & 1.14 \\
20    & 10    & 3     & 0.99  & \textbf{1.01} & 0.32  & \textbf{1.11} & 1.09  & 1.07 \\
20    & 15    & 4     & 0.79  & \textbf{0.81} & 0.26  & \textbf{0.99} & 0.98  & 0.98 \\
20    & 20    & 5     & 0.81  & \textbf{0.82} & 0.33  & \textbf{0.97} & 0.96  & 0.96 \\
30    & 7     & 3     & 1.30  & \textbf{1.31} & 0.52  & \textbf{1.37} & 1.28  & 1.18 \\
30    & 10    & 4     & 0.94  & \textbf{0.96} & 0.30  & \textbf{1.10} & 1.07  & 1.05 \\
30    & 15    & 5     & 0.74  & \textbf{0.76} & 0.26  & \textbf{0.99} & 0.98  & 0.98 \\
30    & 20    & 3     & 0.76  & \textbf{0.78} & 0.26  & \textbf{0.96} & 0.95  & 0.95 \\
\hline
\end{tabular}%
\label{tab:compar_spradii_DPex2HR}%
\end{table}%

\end{example}

\paragraph{On performing the double preconditioning.}
Due to sub-distributivity of interval arithmetic, the evaluation 
$$\textstyle
\inum{H}_1:=A^{(0)}+\sum_{k=1}^K (LA^{(k)}R)\inum{p}_k
$$
gives always as tight interval as the evaluation
$$\textstyle
\inum{H}_2:=A^{(0)}+L\left(\sum_{k=1}^K  A^{(k)}\inum{p}_k\right)R.
$$
In contrast to the left preconditioning (Proposition~\ref{propColOnePrecond}), there is no natural class of matrices, for which $\inum{H}_1=\inum{H}_2$.

\begin{proposition}
Suppose that there are some dependencies in $A(p)$, that is, there is $k$ such that $A^{(k)}_{\alpha\beta},A^{(k)}_{\gamma\delta}\not=0$ for some indices $(\alpha,\beta)\not=(\gamma,\delta)$. Then there are $L,R$ such that $\inum{H}_1\not=\inum{H}_2$.
\end{proposition}

\begin{proof}
The $(i,j)$-th entry of $\inum{H}_1$ is evaluated based on the expression
$$
( L_{i\alpha} A^{(k)}_{\alpha\beta}R_{\beta j} +
  L_{i\gamma} A^{(k)}_{\beta\delta}R_{\delta j} + \dots)\inum{p}_k +\dots,
$$
whereas the $(i,j)$-th entry of $\inum{H}_2$ is evaluated based on the expression
$$
 L_{i\alpha} A^{(k)}_{\alpha\beta}R_{\beta j} \inum{p}_k +
 L_{i\gamma} A^{(k)}_{\beta\delta}R_{\delta j} \inum{p}_k + \dots
$$
Now, it is easy to find $L,R$ such that the former will be a~strict subset of the latter.
\end{proof}

\paragraph{Preconditioning based on decomposition of the midpoint inverse.} If the midpoint matrix is not an identity matrix and is not singular, then its inverse exists and can be decomposed into a~product of two or more matrices. Then these matrices can be used to perform double preconditioning. We consider the following decomposition methods: 
\begin{itemize}
	\item[$-$] LU decomposition: $(\Mid{A})^{-1}=LU$, where $L$ is a~lower triangular matrix with ones on the diagonal and $U$ is an upper triangular matrix. So we precondition from the left by $U$ and from the right by $L$.
	\item[$-$] SVD decomposition: $(\Mid{A})^{-1}=U\Sigma V^T$, where $U$ and $V$ are $n\times n$ orthogonal matrices, and $\Sigma$ is a~diagonal $n\times n$ matrix with singular values of $(\Mid{A})^{-1}$ on the diagonal. So we precondition from the left by $V^T$ and from the right by $U\Sigma$.
	\item[$-$] QR decomposition: $(\Mid{A})^{-1}=QR$, where $Q$ is orthogonal and $R$ is an upper triangular matrix. So we precondition with from the left by $R$ and from the right by $Q$.
\end{itemize}

Preconditioning employing LU decomposition of the midpoint inverse will be referred to as LU preconditioning, the same concerns two other preconditioning strategies.
\begin{example}\normalfont
	\label{ex:randm_nonidmid}
	Consider random interval parametric matrices with non-identity nonsingular midpoint matrix. They are obtained as follows: the elements of the midpoint matrix $A^{(0)}$ are random numbers generalted uniformly from the interval $[-8,8]$, whereas the elements of the $A^{(k)}$ matrices, for $k=1,\ldots,K$, are obtained in the similar manner as in Example \ref{ex:romatrices_1} (in particular they are all of rank one), but this time we draw elements of the random vectors $a_k$ and $b_k$ from the intervals $\inum{u}+[-0.2k,0.2k]$ and $\inum{v}+[-0.2k,0.2k]$, respectively, where $\inum{u}=[-1,2]$ and $\inum{v}=[2,3]$. We compare decomposition based preconditioning (DBP) strategies to each other and DBP combined with S0 strategy. The box-plot for the ratios $\rho(\Rad{A})/\rho(\Rad{H})$ is presented in Fig.~\ref{fig:compar_spradii_DPex5NM}. Additionally, the table with geometric means of the ratios are given in Table~\ref{tab:compar_spradii_DPex5NM}.
\begin{figure}
	\centering
	\includegraphics[width=0.9\textwidth,trim={15pt 15pt 15pt 11pt},clip]{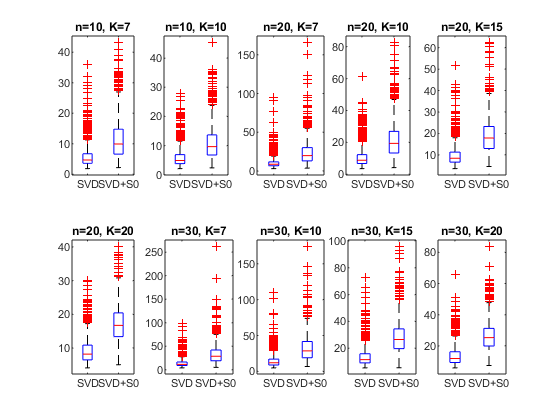}
	\caption{Results for Example~\ref{ex:randm_nonidmid}: boxplot of ratio $\rho(\Rad{A})/\rho(\Rad{H})$; first variant (left), second variant (right); $\inum{u}=[-1,2]$, $\inum{v}=[2,3]$}
	\label{fig:compar_spradii_DPex5NM}
\end{figure}

\begin{table}
	\centering
	\caption{Results for Example~\ref{ex:randm_nonidmid}: the geometric means of $\rho(\Rad{A})/\rho(\Rad{HC})$ for considered preconditioning strategies; the best results are typed in boldface}
	\begin{tabular}{rr|rrrrrr}
		\hline
		\multicolumn{2}{c|}{Sizes} & \multicolumn{6}{c}{First variant} \\
		\hline
		\multicolumn{1}{l}{$n$} & \multicolumn{1}{l|}{$K$} & \multicolumn{1}{c}{ LU} & \multicolumn{1}{c}{ LU+Opt} & \multicolumn{1}{l}{SVD} & \multicolumn{1}{l}{ SVD+Opt} & \multicolumn{1}{l}{QR} & \multicolumn{1}{l}{QR+Opt} \\
		\hline
		10    & 7     & 3.46  & \textbf{9.94} & 5.23  & 9.82  & 4.78  & 10.24 \\
		10    & 10    & 3.52  & \textbf{9.54} & 5.19  & 9.28  & 4.64  & 9.41 \\
		20    & 7     & 5.25  & 19.83 & 9.63  & \textbf{20.19} & 8.36  & 20.08 \\
		20    & 10    & 5.04  & 18.48 & 9.56  & \textbf{19.43} & 7.85  & 18.56 \\
		20    & 15    & 4.86  & \textbf{17.37} & 9.03  & \textbf{17.37} & 7.44  & 16.92 \\
		20    & 20    & 4.64  & 15.95 & 8.71  & \textbf{16.5} & 7.27  & 16.1 \\
		30    & 7     & 6.64  & \textbf{30.01} & 13.45 & 29.48 & 11.44 & 29.21 \\
		30    & 10    & 6.07  & 26.99 & 12.94 & \textbf{27.58} & 11.03 & 27.95 \\
		30    & 15    & 5.95  & 25.77 & 12.51 & \textbf{25.83} & 10.41 & 25.5 \\
		30    & 20    & 5.83  & 23.68 & 12.64 & \textbf{25.06} & 10.19 & 24.24 \\
		\hline
		\multicolumn{2}{c|}{Sizes} & \multicolumn{6}{c}{Second variant} \\
		\hline
		\multicolumn{1}{l}{$n$} & \multicolumn{1}{l|}{$K$} & \multicolumn{1}{l}{ LU} & \multicolumn{1}{l}{ LU+Opt} & \multicolumn{1}{l}{SVD} & \multicolumn{1}{l}{ SVD+Opt} & \multicolumn{1}{l}{QR} & \multicolumn{1}{l}{QR+Opt} \\
		\hline
		10    & 7     & 3.12  & 8.21  & 4.32  & 8.24  & 4.13  & \textbf{8.53} \\
		10    & 10    & 3.16  & \textbf{8.09} & 4.56  & 8.08  & 4.16  & 8.08 \\
		20    & 7     & 4.69  & 15.9  & 7.65  & \textbf{16.14} & 7.03  & 15.91 \\
		20    & 10    & 4.65  & \textbf{16.04} & 7.92  & 15.74 & 7.06  & 15.75 \\
		20    & 15    & 4.54  & 15.32 & 8.23  & \textbf{15.52} & 7.03  & 15.25 \\
		20    & 20    & 4.43  & 15.02 & 8.36  & \textbf{15.49} & 7     & 15.15 \\
		30    & 7     & 5.65  & 24.31 & 10.84 & \textbf{24.53} & 9.92  & 23.6 \\
		30    & 10    & 5.66  & 23.74 & 11.4  & \textbf{23.8} & 10.18 & 23.5 \\
		30    & 15    & 5.50  & 22.47 & 11.79 & \textbf{23.43} & 9.8   & 23.02 \\
		30    & 20    & 5.45  & 22.11 & 12    & \textbf{23.63} & 9.56  & 22.45 \\
		\hline
	\end{tabular}%
	\label{tab:compar_spradii_DPex5NM}%
\end{table}%

As can be seen from the figure and the table, the combination of DBP with S0 strategy for matrices with identity midpoint matrix  significantly decreases the spectral radius. It can be seen as well that the SVD preconditioning seems to be prevailing.
\end{example}

\section{Numerical experiments}\label{sec:numex}
The following examples illustrate the performance of the above proposed preconditioning approaches in the context of solving interval parametric linear systems. All the computation presented below were performed by using authors' own software. The software was implemented in C++ and compiled under Windows 10 using Visual C++ 2017 compiler.

Generally, any method for solving interval parametric linear systems can be adapted to use left, right or double preconditioning. However, based on the results from \cite{HLADIK20191,SkaHla2019NLAA}, we decided to use the Parametric Krawczyk iteration (PKI) with residual correction, which is one of the best methods for solving interval parametric linear systems. For selected examples, we present also the results of the Parametric Hansen-Bliek-Rohn (PHBR) method without residual correction (cf. \cite{Hladik:2012:EFS,HLADIK20191,Skalna:2018:PIAS}), which is a~direct method and which sometimes outperforms PKI. In order to indicate which preconditioning was employed, we add the respective subscript (L -- left preconditioning, R -- right preconditioning, LU -- double LU preconditioning, etc.) to the name of the method.

Given the interval-affine linear system \nref{eq:affintlinsys}, the general scheme of Krawczyk-type iterations (cf.~\cite{HLADIK20191}) can be written as:
\begin{align}\label{iterKraw}
\inum{v}(e)\mapsto \inum{g}(e)+(I-\inum{H}(e))\inum{v}(e),
\end{align}
where $\inum{H}(e)$ and $\inum{g}(e)$ are obtained by a~respective transformation (residual correction and/or left/right/double preconditioning) of \nref{eq:affintlinsys}. The PKI produces a~$p$-solution of the form $\inum{v}(e)=Fe+\inum{a}$. If the right or double preconditioning is involved, then the final solution is $\inum{x}(e)=RFe+R\inum{a}$. Thanks to the fact that $\inum{v}(e)$ partially preserves information about dependencies, the resulting bounds are relatively narrow. 

In order to measure the overestimation of enclosure $\inum{y}$ over $\inum{x}$, where $\inum{x},\inum{y}\in\mathbb{IR}$,  $\inum{x}\subseteq\inum{y}$, we use the following accuracy measure
\begin{equation}
\label{eq:measure1}
O_{\omega}=\left(1-\Rad{x}/\Rad{y}\right)\cdot 100\%.
\end{equation}

\begin{example}\normalfont
\label{ex:ex1}
Consider the following two-dimensional interval parametric linear system with one parameter:
\begin{equation}
\label{eq:ex1_org}
\begin{pmatrix}
p & 2p \\
2 & 1
\end{pmatrix}
\begin{pmatrix}
x_1 \\
x_2
\end{pmatrix}=
\begin{pmatrix}
1 \\
1
\end{pmatrix},\;p\in[1/2,7/2].
\end{equation}
The parametric solution set of the above system as well as the functions $x_1(p)$ and $x_2(p)$ are depicted in Fig.~\ref{fig:fzznum}.

It is not hard to verify that $\rho\left(\Rad{B}\right)=1.25>1$ and $\rho(B'\Rad{})=0.75<1$. This means that $A(\inum{p})$ is regular and thus the parametric solution set is bounded. However, as already mentioned, most of the existing methods for solving interval parametric linear systems will fail to solve the system since they require that $\rho(\Rad{B})<1$. Using our $p$-solution based approach we are able to solve the system \nref{eq:ex1_org} and obtain relatively narrow enclosures (see Table~\ref{tab:ex1tab}).

\begin{figure}
\centering
\begin{tikzpicture}[scale=1.0] 
\begin{axis}[
scale=0.43,
xmin=0.5, xmax=3.5,
font=\scriptsize,
xlabel=$p$,
ylabel={$x_1(p)$},samples=100,
xtick={0.5,3.5}
]
\addplot[no marks] {(1-2*x)/(-3*x)};
\end{axis}
\begin{scope}[xshift=4cm]
\begin{axis}[
scale=0.43,
xmin=0.5, xmax=3.5,
font=\scriptsize,
xlabel=$p$,
ylabel={$x_2(p)$},samples=100,
xtick={0.5,3.5}
]
\addplot[no marks] {(x-2)/(-3*x)};
\end{axis}
\end{scope}
\begin{scope}[xshift=8cm]
\begin{axis}
[
scale=0.43,
xmin=0,   xmax=0.6,
ymin=-0.15,   ymax=1,
xlabel=$x_1(p)$,ylabel=$x_2(p)$,
font=\scriptsize
]
\addplot+[black,solid,no marks,domain=0.5:3.5] ({(2*\x-1)/(3*\x)},{(2-\x)/(3*\x))});
\end{axis}
\node at (1.2,0.7) {$S(p)$};
\end{scope}
\end{tikzpicture}
\caption{The elements of the solution set of \nref{eq:ex1_org} as functions of parameter $p$ (since $x_1$ and $x_2$ are monotone with respect to $p$, their extremal values are attained at respective endpoints of $\inum{p}$) and the solution set of system~\nref{eq:ex1_org}; Example~\ref{ex:ex1}}
\label{fig:fzznum}
\end{figure}
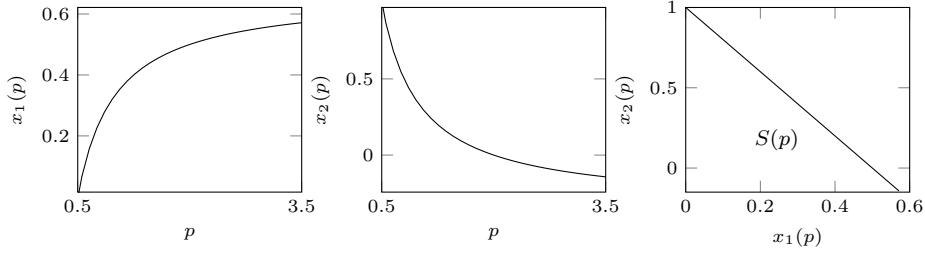

\begin{table}
\centering
\caption{Results for Example~\ref{ex:ex1}: comparison of interval enclosures produced by PHBR$_\textrm{R}$ and PKI$_\textrm{R}$; the third row shows an inner estimation of the inteval hull produced by PKI$_\textrm{R}$}
\label{tab:ex1tab}
\begin{tabular}{l|rr}
\hline
Method & $x_1$ & $x_2$ \\
\hline
PHBR$_\textrm{R}$ & $[-9.992007\cdot 10^{-16},0.571429]$ & $[-0.142857,1]$ \\
PKI$_\textrm{R}$ & $[-1.693090\cdot 10^{-15},0.740747]$ & $[-0.481481,1]$ \\
PKI$_\textrm{R}$ inner & $[0.296296, 0.444444]$ & $[0.111111, 0.407407]$ \\
hull & $[0,12/21]$ & $[-3/21,1]$ \\
\hline
\end{tabular}
\end{table}

As we can see from the table, the PHBR$_\textrm{R}$ method with right preconditioning produced narrower bounds than PKI$_\textrm{R}$, in fact it produced the hull of the parametric solution set. However, the PKI$_\textrm{R}$ produced additionally an inner estimation (i.e., a~subset) of the interval hull of the solution set. Notice that an inner estimation is useful for the comparison purposes, i.e., whenever the hull is not known, it enables us to judge the quality of outer bounds.

\end{example}

\begin{example}\normalfont
	\label{ex:ex2}
	Consider the following three-dimensional interval parametric linear system with three parameters:
	\begin{equation}
	\begin{pmatrix}
		1 + p_1 - p_2 & -p_1 + p_2 & 1 + p_1 \\
		2 + p_1 + p_2 & -1 - p_1 - p_2 + p_3 & -1 + p_1 - 2p_3 \\
		1 + p_1 & -3 - p_1 - 2p_3 & 6 + p_1 + 4p_3
	\end{pmatrix}x=
	\begin{pmatrix}
		1\\
		1\\
		1
	\end{pmatrix},\ 
	\begin{array}{ll}
		p_1 \in [-\delta,\delta] \cr
		p_2 \in [-\delta,\delta] \cr
		p_3 \in [-\delta,\delta]
	\end{array}
	\end{equation}
The values of $\rho(\Rad{H})$ (see formula \nref{eq:b_and_bprime2_lincase}) for $\delta$ ranging from $0.01$ to $0.6$ and for $\inum{H}$ obtained using, respectively, left ($L=(\Mid{A})^{-1}$, $R=I_n$), right ($L=I_n$, $R=(\Mid{A})^{-1}$), and double preconditioning ($RL=(\Mid{A})^{-1}$) are presented in Fig.~\ref{fig:ex2_spectr_radii}. As can be seen from the figure, double LU preconditioning produced the smallest spectral radius. In particular, using double LU preconditioning, the problem can be solved for $\delta\leqslant 0.55$, whereas the classical methods (that use left preconditioning) will fail for $\delta\geqslant 0.27$.
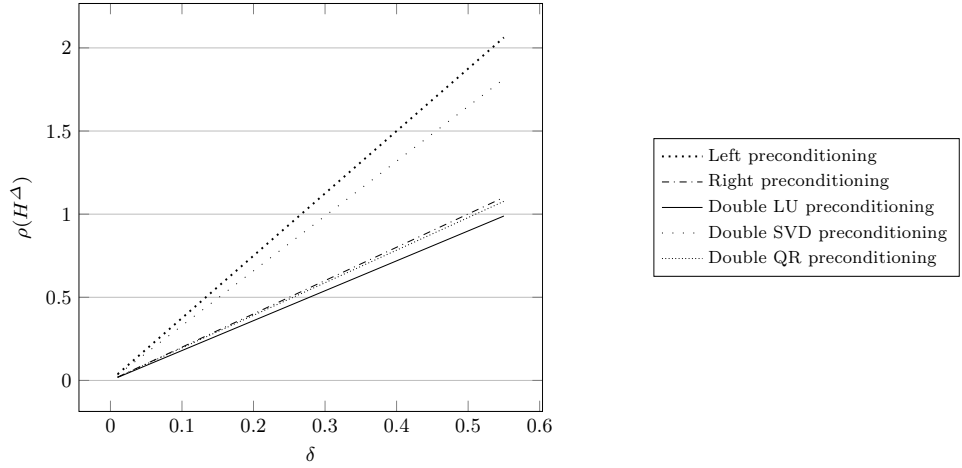
\begin{figure}
	\centering
	\begin{tikzpicture}[xscale=0.89,yscale=0.95,
	axis/.style={very thick, ->, >=stealth'},
	important line/.style={thick},
	dashed line/.style={dashed, thick},
	every node/.style={color=black,}
	]
	\begin{axis}[
	ylabel=$\rho(\Rad{H})$,
	xlabel={$\delta$},
	ymajorgrids=true,
	legend style={at={(1.9,0.5)},anchor=east},
	legend style={font=\scriptsize},
	legend cell align={left}
	]
	\addplot[thick,dotted,smooth] coordinates {(0.01,0.0375012) (0.55,2.06256)};
	\addplot[dashdotted,smooth] coordinates {(0.01,0.0200213) (0.55,1.10117)};
	\addplot[solid,smooth] coordinates {(0.01,0.0179797) (0.55,0.988881)};
	\addplot[loosely dotted,smooth] coordinates {(0.01,0.0329613) (0.55,1.81287)};
	\addplot[densely dotted,smooth] coordinates {(0.01,0.0195877) (0.55,1.07732)};
	\addlegendentry{Left preconditioning}
	\addlegendentry{Right preconditioning}
	\addlegendentry{Double LU preconditioning}
	\addlegendentry{Double SVD preconditioning}
	\addlegendentry{Double QR preconditioning}
	\end{axis}
	\end{tikzpicture}
	\caption{Results for Example~\ref{ex:ex2}: comparison of $\rho\left(\Rad{H}\right)$ as a~function of~$\delta$ with $\inum{H}$ obtained by using various preconditioning approaches}
	\label{fig:ex2_spectr_radii}
\end{figure}

However, the results are not straightforward (see Table~\ref{tab:ex2_overerst}), the PKI$_{\textrm{LU}}$ produced the best bounds only for $x_1$, whereas the PKI$_{\textrm{QR}}$ produced the best bounds for the remaining entries. The PKI$_L$ produced the worse bounds. Table~\ref{tab:ex2_overerst} reports the overestimation of the hull computed by using the formula \nref{eq:measure1}.
\end{example}

\begin{table}
	\centering
	\caption{Results for Example~\ref{ex:ex2}: overestimation of the interval hull  produced by Parametric Krawczyk iteration with various preconditioning approaches}
	\label{tab:ex2_overerst}
	\begin{tabular}{l|rrrr}
		\hline
	Method & $\delta=5\%$ & $\delta = 10\%$ & $\delta=20\%$ \\
	\hline
	\multirow{3}{*}{PKI$_\textrm{L}$} & 11\% & 23\% & 51\% \\
	& 10\% & 21\% & 48\% \\
	& 13\% & 27\% & 55\% \\
	\hline
	\multirow{3}{*}{PKI$_\textrm{R}$} & 6\% & 12\% & 26\% \\
	& 8\% & 17\% & 34\% \\
	& 7\% & 15\% & 31\% \\
	\hline
	\multirow{3}{*}{PKI$_\textrm{LU}$} & \textbf{3\%} & \textbf{7\%} & \textbf{16\%} \\
	& 9\% & 18\% & 35\% \\
	& 10\% & 19\% & 37\% \\
	\hline
	\multirow{3}{*}{PKI$_\textrm{SVD}$} & 6\% & 12\% & 26\% \\
	& 7\% & 14\% & 29\% \\
	& 8\% & 15\% & 31\% \\
	\hline
	\multirow{3}{*}{PKI$_\textrm{QR}$} & 6\% & 12\% & 26\% \\
	& \textbf{6\%} & \textbf{12\%} & \textbf{26\%} \\
	& \textbf{7\%} & \textbf{14\%} & \textbf{29\%} \\
	\hline
	\end{tabular}
\end{table}

\begin{example}\normalfont
\label{ex:ex3}
Consider the following three-dimensional interval parametric linear system with three parameters (cf.~\cite{Pop2018a}):
\begin{figure}
\centering
\includegraphics[width=0.49\textwidth]{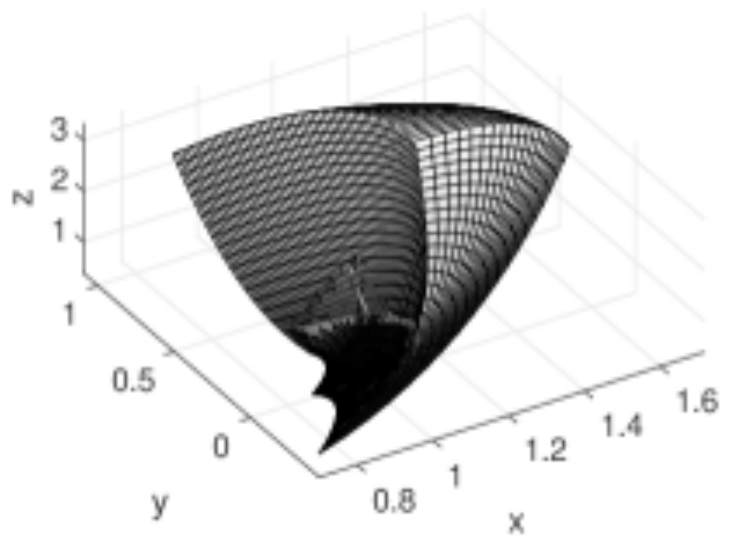}
\includegraphics[width=0.49\textwidth]{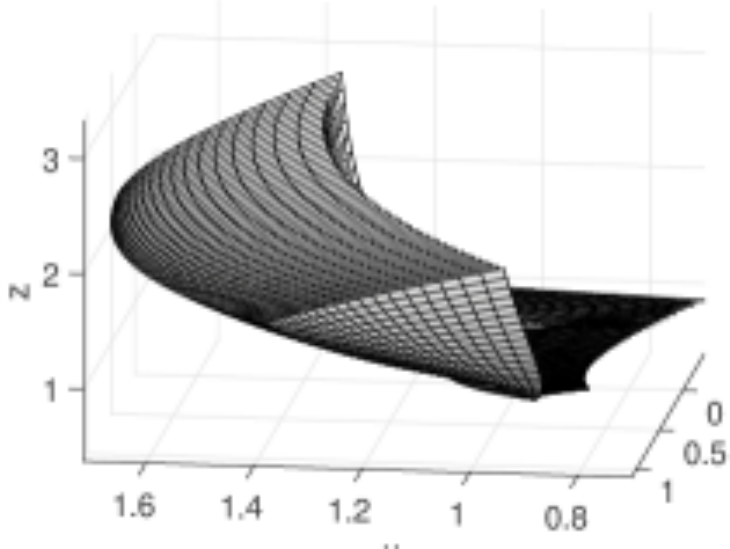}
\caption{Solution set of system (\ref{eq:sys2}) viewed from different perspectives; Example~\ref{ex:ex3}}
\label{fig:exnum2solset}
\end{figure}
\begin{equation}
\label{eq:sys2}
\begin{pmatrix}
\frac{1}{2}-p_2 & p_1 & p_1 \\
p_2 & -p_2 & p_3 \\
p_1 & p_3 & 1
\end{pmatrix}
x=
\begin{pmatrix}
p_2 \\
2p_2 \\
3p_2
\end{pmatrix},\;
\begin{array}{l}
p_1 \in [\frac{3}{4},\frac{5}{4}] \\[1pt]
p_2 \in [\frac{1}{2},\frac{3}{2}] \\[1pt]
p_3 \in [\frac{1}{2},\frac{3}{2}]
\end{array}
\end{equation}
In this case, left preconditioning (with $L=(\Mid{A})^{-1}$ and $R=I_n$ in~\nref{eq:strreg_dbl}) gives $\rho(\Rad{H})\approx 1.12$, which means that most of the existing methods will fail to solve the system (\ref{eq:sys2}) (since they rely on left preconditioning).
On the other hand, right preconditioning (we put $L=I_n$ and $R=(\Mid{A})^{-1}$ in~\nref{eq:strreg_dbl}) yields
$\rho(\Rad{H})\approx 0.97<1$. So, the proposed here approach, which employs right preconditioning, is applicable. The obtained results together with the result from \cite{Pop2018a} are presented in Table \ref{tab:ex2tab}. As can be seen from the table, the result from \cite{Pop2018a} is quite rough (cf. the solution set of the system (\ref{eq:sys2}) depicted in Fig.~\ref{fig:exnum2solset}), but was so far the only solution available for the considered system. Using the PKI$_\textrm{R}$ method we have obtained the solution, which is still quite rough, but is significantly better than the enclosure from \cite{Pop2018a}. As can be seen, the latter overestimates the obtained here enclosure by $(58\%,58\%,59\%)$.
\begin{table}
\centering
\caption{Results for Example~\ref{ex:ex3}: interval enclosures for parametric solution set of system \nref{eq:sys2}}
\label{tab:ex2tab}
\begin{tabular}{rrr}
\hline
PKI$_\textrm{R}$ & \cite{Pop2018a} & interval hull \\
\hline
$[-16.768697, 18.556510]$ & $[-41.11159,43.77826]$ & $[0.69999,1.7157]$ \\
$[-18.197915, 18.535419]$ & $[-43.11161,44.11161]$ & $[-0.4501,1.0938]$ \\
$[-20.214964, 23.743957]$ & $[-51.88948,54.22282]$ & $[0.3818,3.3244]$ \\
\hline
\end{tabular}
\end{table}
\end{example}

\begin{example}\normalfont
\label{ex:ex5}
Consider the following interval parametric linear system (cf.~\cite{Okumura:1993:AIO})
\begin{equation}
\begin{pmatrix}
p_1 + p_6 & -p_6 & 0 & 0 & 0 \\
-p_6 & p_2 + p_6 + p_7 & -p_7 & 0 & 0 \\
0 & -p_7 & p_3 + p_7 + p_8 & -p_8 & 0 \\
0 & 0 & -p_8 & p_4 + p_8 + p_9 & -p_9 \\
0 & 0 & 0 & -p_9 & p_5 + p_9
\end{pmatrix}
x=
\begin{pmatrix}
10 \\
0 \\
10 \\
0 \\
0
\end{pmatrix}.
\end{equation}
The nominal values of all parameters are equal to 1. We solve the system with parameter tolerances 10\%, 20\% and 30\%. The PKI$_\textrm{L}$ and PKI$_\textrm{LU}$ produced the best results. The overestimation of PKI$_\textrm{L}$ enclosures over PKI$_\textrm{LU}$ enclosures are reported in Table~\ref{tab:exmh_results}. As can be seen from the table, for 10\% tolerance, the double LU preconditioning improved the bounds for $x_1$--$x_3$, whereas the bounds for $x_4$ and $x_5$ got worse. For 20\% and 30\% tolerance, the double LU preconditioning improved all bounds, for 30\% tolerance the improvement is quite large.

\begin{table}
	\centering
	\caption{Results for Example~\ref{ex:ex5}: overestimation of PKI$_\textrm{L}$ over PKI$_\textrm{LU}$ enclosures}
	\begin{tabular}{c|rrr}
		\hline
		$x$  & $\delta=$10\% &$\delta=$ 20\% & $\delta=$30\% \\
		\hline
		$x_1$ & 3.6\% & 9.0\% & 19.6\%  \\
		$x_2$ & 1.7\% & 5.7\% & 15.2\%  \\
		$x_3$ & 0.3\% & 2.7\% & 12.1\%  \\
		$x_4$ & -0.4\% & 2.6\% & 12.8\%  \\
		$x_5$ & -1.5\% & 0.6\% & 10.4\%  \\
		\hline
	\end{tabular}%
	\label{tab:exmh_results}%
\end{table}
\end{example}

\begin{example}\normalfont
	\label{ex:electrical_circuit_cmplx}
	Consider interval parametric linear system (\ref{eq:circuit_ac}), which occurs in worst-case tolerance analysis of linear AC (alternate current) electrical circuits~\cite{Kolev:1993:IMCA,Kolev:2002:WCA,Zimmer:2012:SVSPLS}. 
	The circuit studied is shown in Fig.~\ref{fig:lin_dc_circuit} (cf.\ Kolev~\cite{Kolev:1993:IMCA}). It has eleven branches and five nodes (not including the datum node). The parameters of the model have the following nominal values:
\begin{flalign*}
	& e_1=e_2=100V,\,e_5=e_7=10V, \\
	& Z_j=R_j+iX_j\in\mathbb{C},\,R_j=100\Omega,\,X_j=\omega L_j-\frac{1}{\omega C_j},\,j=1,\ldots,11, \\
	& \omega=50,\,X_{1,2,5,7}=\omega L_{1,2,5,7}=20, X_3=\omega L_3=30, \\
	& X_4 =-\frac{1}{\omega C_4}=-300,\,X_{10}=-\frac{1}{\omega C_{10}}=-400,\,X_{6,8,9,11}=0.
\end{flalign*}
The electric parameters resistance $R_j$, inductance $L_j$, and capacitance $C_j$, $j=1,\ldots,11$, of the branch elements are considered to be unknown but vary within given intervals. The amplitudes $e_1,e_2,e_5,e_7$ of the sine voltages are assumed to have zero tolerances. The goal here is to find bounds for the real and imaginary parts of the node voltages $V_1,\ldots,V_5$.
	
	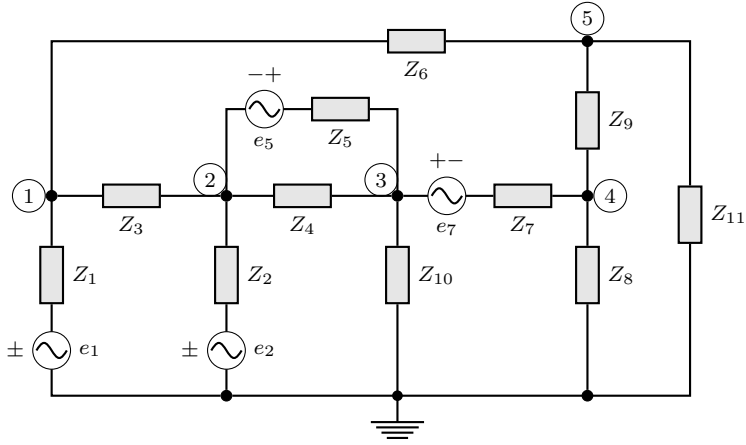
\begin{figure}
		\centering
		\begin{tikzpicture}[
		mynode/.style={circle,minimum width=4pt,fill=white,draw,inner sep=2}, 
		mypoint/.style={circle,minimum width=4pt,fill=black,draw,inner sep=0}, 
		blockh/.style={draw,thick,fill=gray!20,rectangle,minimum height=1em,minimum width=2.5em},
		blockv/.style={draw,thick,fill=gray!20,rectangle,minimum width=1em,minimum height=2.5em}
	]
	\node[blockh,label=below:$Z_6$] (z6) at (4.5,4.8) {};
	\node[blockh,label=below:$Z_5$] (z5) at (3.5,3.9) {};
	\node[blockh,label=below:$Z_3$] (z3) at (0.75,2.75) {};
	\node[blockh,label=below:$Z_4$] (z4) at (3,2.75) {};
	\node[blockh,label=below:$Z_7$] (z7) at (5.9,2.75) {};
	\node[blockv,label=right:$Z_1$] (z1) at (-0.3,1.7) {};
	\node[blockv,label=right:$Z_2$] (z2) at (2,1.7) {};
	\node[blockv,label=right:$Z_{10}$] (z10) at (4.25,1.7) {};
	\node[blockv,label=right:$Z_8$] (z8) at (6.75,1.7) {};
	\node[blockv,label=right:$Z_9$] (z9) at (6.75,3.75) {};
	\node[blockv,label=right:$Z_{11}$] (z11) at (8.1,2.5) {};
	\draw[thick] (4.25,0.1) node[ground] (gr0) {} node[circ](4.5,0.1) {};
	\node[mypoint] (p5) at (6.75,4.8) {};
	\node[mypoint] (gr1) at (6.75,0.1) {};
	\node[mypoint] (gr2) at (2,0.1) {};
	\node[mypoint] (p1) at (-0.3,2.75) {};
	\node[mypoint] (p2) at (2,2.75) {};
	\node[mypoint] (p3) at (4.25,2.75) {};
	\node[mypoint] (p4) at (6.75,2.75) {};
	\node[mynode,above=0cm of p5] {\footnotesize{5}};
	\node[mynode,left=0cm of p1] {\footnotesize{1}};
	\node[mynode,right=0cm of p4] {\footnotesize{4}};
	\node[mynode,above left=0cm and 0cm of p2] {\footnotesize{2}};
	\node[mynode,above left=0cm and 0cm of p3] {\footnotesize{3}};
	\draw[thick] (z1) |- (z6) -| (z11) |- (gr2) -| (z1);
	\draw[thick] (p1) -- (z3) -- (p2) |- (z5) -| (p3) -- (z7) -| (z9) -- (p5);
	\draw[thick] (gr2) -- (z2) -- (p2) -- (z4) -- (p3) -- (z10) -- (gr0);
	\draw[thick] (p4) -- (z8) -- (gr1);
	\node[mynode,label=left:$\pm$,label=right:$e_1$,inner sep=5pt] (e1) at (-0.3,0.7) {};
	\draw[thick] (-0.5,0.7) sin (-0.4,0.8) cos (-0.3,0.7) sin (-0.2,0.6) cos (-0.1,0.7);
	\node[mynode,label=left:$\pm$,label=right:$e_2$,inner sep=5pt] (e2) at (2,0.7) {};
	\draw[thick] (1.8,0.7) sin (1.9,0.8) cos (2,0.7) sin (2.1,0.6) cos (2.2,0.7);
	\node[mynode,label=above:$+-$,label=below:$e_7$,inner sep=5pt] (e7) at (4.9,2.75) {};
	\draw[thick] (4.7,2.75) sin (4.8,2.85) cos (4.9,2.75) sin (5.0,2.65) cos (5.1,2.75);
	\node[mynode,label=above:$-+$,label=below:$e_5$,inner sep=5pt] (e5) at (2.5,3.9) {};
	\draw[thick] (2.3,3.9) sin (2.4,4) cos (2.5,3.9) sin (2.6,3.8) cos (2.7,3.9);
\end{tikzpicture}
		\caption{Linear electrical AC circuit with five nodes and eleven branches; Example~\ref{ex:electrical_circuit_cmplx}}
		\label{fig:lin_dc_circuit}
	\end{figure}
	
	The nodal analysis of the considered circuit leads to the following complex parametric linear system~\cite{Kolev:2002:WCA,PopKolKram:2010:SCVPLS}:
	\begin{flalign}
	\label{eq:circuit_ac}
	& \left(
	\begin{array}{ccc}
	\frac{1}{Z_1}+\frac{1}{Z_3}+\frac{1}{Z_6} & -\frac{1}{Z_3} & 0 \\[2pt]
	-\frac{1}{Z_3} & \frac{1}{Z_2}+\frac{1}{Z_3}+\frac{1}{Z_4}+\frac{1}{Z_5} & -\frac{1}{Z_4}-\frac{1}{Z_5} \\[2pt]
	0 & -\frac{1}{Z_4}-\frac{1}{Z_5} & \frac{1}{Z_4}+\frac{1}{Z_5}+\frac{1}{Z_7}+
	\frac{1}{Z_{10}} \\[2pt]
	0 & 0 & -\frac{1}{Z_7} \\[2pt]
	-\frac{1}{Z_6} & 0 & 0
	\end{array}
	\right. \\
	& \quad\quad\quad\quad\quad\quad\quad\quad\quad\left.
	\begin{array}{cc}
	0 & -\frac{1}{Z_6} \\[2pt]
	0 & 0 \\[2pt]
	-\frac{1}{Z_7} & 0 \\[2pt]
	\frac{1}{Z_7}+\frac{1}{Z_8}+\frac{1}{Z_9} & -\frac{1}{Z_9} \\[2pt]
	-\frac{1}{Z_9} & \frac{1}{Z_6}+\frac{1}{Z_9}+\frac{1}{Z_{11}}
	\end{array}
	\right)
	\begin{pmatrix}
	V_1 \\[2pt]
	V_2 \\[2pt]
	V_3 \\[2pt]
	V_4 \\[2pt]
	V_5
	\end{pmatrix}
	=
	\begin{pmatrix}
	\frac{e_1}{Z_1} \\[2pt]
	\frac{e_2}{Z_2}-\frac{e_5}{Z_5} \\[2pt]
	\frac{e_5}{Z_5}+\frac{e_7}{Z_7} \\[2pt]
	-\frac{e_7}{Z_7} \\[2pt]
	0
	\end{pmatrix}
	\nonumber
	\end{flalign}
	
	Without loss of generality we change the system parameters and substitute $p_j=1/Z_j$, $j=1,\ldots,11$. This way the parametric system involves affine-linear dependencies in the matrix. The system \nref{eq:circuit_ac} is then replaced with an equivalent twice larger real parametric linear system with $18$ real parameters~\cite{Hladk2010SSO,Skalna:2018:PIAS}.
	%
	%
	We solve the latter system with parameter tolerances $\pm 5\%$, $\pm 10\%$, $20\%$, and $25\%$ by using the PKI method with DBP. The best results were produced by the PKI$_\textrm{L}$ (standard approach) and PKI$_\textrm{LU}$ methods, and thus their results are compared in terms of accuracy. Table~\ref{tab:ac_circ_iagsi} shows the overestimation of the PKI$_\textrm{L}$ over PKI$_\textrm{LU}$ enclosure by means of formula~\nref{eq:measure1}.
	
	\begin{table}
		\centering
		\caption{Results for Example~\ref{ex:electrical_circuit_cmplx}: overestimation of PKI$_\textrm{L}$ over PKI$_\textrm{LU}$ enclosure}
		\begin{tabular}{r|rr|rr|rr|rr}
			\hline
			\multirow{2}{*}{Voltage} & \multicolumn{2}{c|}{5\%} & \multicolumn{2}{c|}{10\%} & \multicolumn{2}{c|}{20\%} & \multicolumn{2}{c}{25\%} \\
			\cline{2-9}
			& Re & Im & Re & Im & Re & Im & Re & Im \\
			\hline
			$V_1$ & 3\% & 2\% & 7\% & 6\% & 26\% & 25\% & 58\% & 58\% \\
			$V_2$ & 2\% & 1\% & 5\% & 3\% & 25\% & 21\% & 58\% & 56\% \\
			$V_3$ & 3\% & 1\% & 8\% & 3\% & 29\% & 22\% & 60\% & 57\% \\
			$V_4$ & 0\% & -4\% & 0\% & -6\% & 14\% & 8\% & 48\% & 47\% \\
			$V_5$ & -2\% & -6\% & -2\% & -9\% & 8\% & 1\% & 41\% & 42\% \\
			\hline
			\#iter$_\textrm{L}$ & \multicolumn{2}{c|}{4} & \multicolumn{2}{c|}{5} & \multicolumn{2}{c|}{13} & \multicolumn{2}{c}{36} \\
			\hline
			\#iter$_\textrm{LU}$ & \multicolumn{2}{c|}{4} & \multicolumn{2}{c|}{5} & \multicolumn{2}{c|}{9} & \multicolumn{2}{c}{16} \\
\hline
		\end{tabular}%
		\label{tab:ac_circ_iagsi}%
	\end{table}%
	For $5\%$ and $10\%$ tolerances, the use of the double LU preconditioning improved the bounds for $V_1$--$V_3$ voltages, but the bounds for $V_4$ and $V_5$ voltages got worse. For $20\%$ and $25\%$ tolerance, the use of the double LU preconditioning improved all bounds, in particular for $V_1$--$V_3$ voltages the improvement was quite large. Moreover, for the two largest tolerances, the PKI$_\textrm{LU}$ was less time consuming (it converged much faster).
\end{example}

\begin{example}\normalfont
	\label{ex:frame}
	Consider a~simple one-bay structural steel frame, shown in Fig.~\ref{fig:frame}, which was initially analyzed by Corliss \textit{et al.}~\cite{Corliss:2007:FRASF}. By applying conventional methods for frame structures analysis, the following parametric linear system is obtained~\cite{Corliss:2007:FRASF,Popova:2007:SLSW}.
	
	\begin{figure}
		\centering
		\begin{tikzpicture}[
	blockh/.style={draw=gray!50,thick,fill=gray!50,rectangle,minimum height=0.7em,minimum width=2.5em},
	mypoint/.style={circle,minimum width=1pt,fill=black,draw,inner sep=0}, 
	]
	\draw[very thick] (0,0) -- (0,3.5);
	\draw[very thick] (6,0) -- (6,3.5);
	\node[blockh] (s1) at (0,0) {};
	\node[blockh] (s2) at (6,0) {};
	\draw[very thick] (s1.north west) -- (s1.north east);
	\draw[very thick] (s2.north west) -- (s2.north east);
	\node[mypoint] (p0) at (0,3.5) {};
	\node[mypoint] (p1) at (0.5,3.5) {};
	\node[mypoint] (p1a) at (0.2,3.5) {};
	\node[mypoint] (p2) at (1,3.5) {};
	\node[mypoint] (p3) at (5,3.5) {};
	\node[mypoint] (p4) at (5.5,3.5) {};
	\node[mypoint] (p4a) at (5.8,3.5) {};
	\node[mypoint] (p5) at (6,3.5) {};
	\node (p12) at (0.75,3.1) {};
	\node (p22) at (5.55,3.1) {};
	\draw[very thick] (p2) -- (p3);
	\draw[very thick] (p0) -- (p1);
	\draw[very thick] (p4) -- (p5);
	\node[very thick,label=below:{$E_b$\quad$I_b$\quad$A_b$}] at (3,3.5) {};
	\node[very thick,label=left:{$E_c$}] at (0,2.4) {};
	\node[very thick,label=left:{$I_c$}] at (0,1.8) {};
	\node[very thick,label=left:{$A_c$}] at (0,1.3) {};
	\node[very thick,label=left:{$E_c$}] at (6,2.4) {};
	\node[very thick,label=left:{$I_c$}] at (6,1.8) {};
	\node[very thick,label=left:{$A_c$}] at (6,1.3) {};
	\draw [very thick] (p1a) to[out=270,in=180] (p12) to[out=0,in=270] (p2);
	\draw [very thick] (p1a) to[out=90,in=90] (p1);
	\draw [very thick] (p3) to[out=270,in=180] (p22) to[out=0,in=270] (p4a);
	\draw [very thick] (p4) to[out=90,in=90] (p4a);
	\draw (0,3.8) -- (0,4.2);
	\draw (6,3.8) -- (6,4.2);
	\draw (6.5,0.15) -- (6.9,0.15);
	\draw (6.5,3.5) -- (6.9,3.5);
	\draw[triangle 45-] (0,4) to (2.5,4);
	\draw[-triangle 45] (3.5,4) to (6,4);
	\node at (3,4) {$L_b$};
	\draw[-triangle 45] (6.7,2.2) to (6.7,3.5);
	\draw[triangle 45-] (6.7,0.15) to (6.7,1.5);
	\node at (6.7,1.85) {$L_c$};
\end{tikzpicture}
		\caption[One-bay structural steel frame]{One-bay structural steel frame \cite{Corliss:2007:FRASF,Popova:2007:SLSW}; Example~\ref{ex:frame}}
		\label{fig:frame}
	\end{figure}
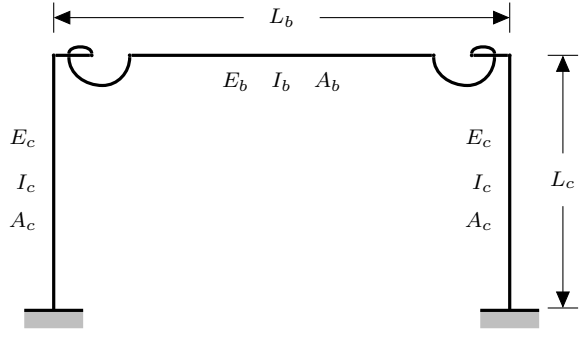
	
	\begin{flalign}
	\label{eq:frame_system}
	& \left(
	\begin{array}{cccccc}
	\frac{A_bE_b}{L_b}+\frac{12E_cI_c}{L^3_c} & 0 & \frac{6E_cI_c}{L^2_c} & 0 & 0 \\[2pt]
	0 & \frac{A_cE_c}{L_c}+\frac{12E_bI_b}{L^3_b} & 0 & \frac{6E_bI_b}{L^2_b} & \frac{6E_bI_b}{L^2_b} \\[2pt] 
	\frac{6E_cI_c}{L^2_c} & 0 & \alpha+\frac{4E_cI_c}{L_c} & -\alpha & 0 \\[2pt]
	0 & \frac{6E_bI_b}{L^2_b} & -\alpha & \alpha+\frac{4E_bI_b}{L_b} & \frac{2E_bI_b}{L_b} \\[2pt]
	0 & \frac{6E_bI_b}{L^2_b} & 0 & \frac{2E_bI_b}{L_b} & \alpha+\frac{4E_cI_c}{L_c} \\[2pt]
	-\frac{A_bE_b}{L_b} & 0 & 0 & 0 & 0 \\[2pt]
	0 & -\frac{12E_bI_b}{L^3_b} & 0 & -\frac{6E_bI_b}{L^2_b} & -\frac{6E_bI_b}{L^2_b} \\[2pt]
	0
	\end{array}
	\right. \\
	&
	\quad\quad\quad\quad\left.
	\begin{array}{cccccc}
	& & & -\frac{A_bE_b}{L_b} & 0 & 0 \\[2pt]
	& & & 0 & -\frac{12E_bI_b}{L^3_b} & 0 \\[2pt]
	& & & 0 & 0 & 0 \\[2pt]
	& & & 0 & -\frac{6E_bI_b}{L^2_b} & 0 \\[2pt]
	& & & 0 & -\frac{6E_bI_b}{L^2_b} & -\alpha \\[2pt]
	& & & \frac{A_bE_b}{L_b}+\frac{12E_cI_c}{L^3_c} & 0 & \frac{6E_cI_c}{L^2_c} \\[2pt]
	& & & 0 & \frac{A_cE_c}{L_c}+\frac{12E_bI_b}{L^3_b} & -\frac{6E_bI_b}{L^2_b}  \\[2pt]
	& & & \frac{6E_cI_c}{L^2_c} & -\frac{6E_bI_b}{L^2_b} & \alpha+\frac{4E_cI_c}{L_c}
	\end{array}
	\right)
	\begin{pmatrix}
	d2_x \\
	d2_y \\
	r2_z \\
	r5_z \\
	r6_z \\
	d3_x \\
	d3_y \\
	r3_z
	\end{pmatrix}=
	\begin{pmatrix}
	H \\
	0 \\
	0 \\
	0 \\
	0 \\
	0 \\
	0 \\
	0
	\end{pmatrix}.\nonumber
	\end{flalign}
	The elements of the system (\ref{eq:frame_system}) are rational functions of Young modulus $E_b$, $E_c$, second moment of area $I_b$, $I_c$, cross-sectional area $A_b$, $A_c$, length $L_b$, $L_c$ and joint stiffness $\alpha$. The right-hand side vector depends on the horizontal force $H$ only. In Corliss et al.~\cite{Corliss:2007:FRASF}, all the parameters, except the lengths, were assumed to be uncertain and varying within given intervals. The nominal values of the model parameters and the worst case uncertainties are given in Table~\ref{tab:frame_param}.
\begin{table}
	\centering
	\caption{Parameters of one-bay structural steel frame: nominal values and worst case uncertainties; Example~\ref{ex:frame}}
	\begin{tabular}{l|l|r}
		\hline
		Parameter & Nominal value & Uncertainty \\
		\hline
		$E_b$, $E_c$ & $29\cdot 10^6$ lbs/in$^2$ & $\pm 348\cdot 10^4$ \\
		$I_b$ & 510 in$^4$ & $\pm 51$ \\
		$I_c$ & 272 in$^4$ & $\pm 27.2$ \\
		$A_b$ & 10.3 in$^2$ & $\pm 1.3$ \\
		$A_c$ & 14.4 in$^2$ & $\pm 1.44$ \\
		$H$   & 5305.5 lbs & $\pm 2203.5$ \\
		$\alpha$ & $2.77461\cdot 10^8$ lb-in/rad & $\pm 1.26504\cdot 10^8$ \\
		\hline
		$L_b$ & 288 in &  \\
		$L_c$ & 144 in &  \\
		\cline{1-2}
	\end{tabular}
	\label{tab:frame_param}
\end{table}

	In order to compare the preconditioning strategies, we solved the system (\ref{eq:frame_system}) with parameter uncertainties, which are 10\%, 20\% and 30\% of the values from the last column of Table~\ref{tab:frame_param}. The best results were produced by using the left and double LU preconditioning. Table~\ref{tab:frame_results} reports the overestimation of PKI$_\textrm{L}$ and PHBR$_\textrm{LU}$ enclosures over PKI$_\textrm{LU}$ enclosure (minus means that PKI$_\textrm{LU}$ overestimates given bounds). For 10\% tolerance, double LU preconditioning improved the PKI$_\textrm{L}$ bounds for 6 out of 8 solution components, the remaining two bounds were a~bit worse. For 20\% tolerance double LU preconditioning improved 7 out of 8 bounds, only the bound for $d3_y$ was slightly worse. For 30\% tolerance double LU preconditioning improved all PKI$_\textrm{L}$ bounds, 6 out of 8 were improved to large extent. Regarding the PHBR$_\textrm{LU}$ method, which often produce rather poor results, it turned out to be useful in this case. For 10\%, 20\% and 30\% tolerances it produced the best bounds for the solution components $d2_x$ and $d3_x$. For 40\% tolerance the PHBR$_\textrm{LU}$ produced the best bounds for 4 out of 8 solution components, whereas for 50\% tolerance the PHBR$_\textrm{LU}$ method produced the best results. Since the PHBR$_\textrm{LU}$ improves the lower bound, it seems reasonable to combine its results with the results of PKI$_\textrm{LU}$, thus obtaining even better bounds.
	\begin{table}
		\centering
		\caption{Results for Example~\ref{ex:frame}: overestimation of PKI$_\textrm{L}$ over PKI$_\textrm{LU}$ enclosure and PHBR$_\textrm{LU}$ over PKI$_\textrm{LU}$ enclosure}
		\begin{tabular}{r|rrrrr|rrrrr}
			\hline
			Solution & \multicolumn{5}{c|}{PKI$_\textrm{L}$ vs PKI$_\textrm{LU}$} & \multicolumn{5}{c}{PHBR$_\textrm{LU}$ vs PKI$_\textrm{LU}$} \\
			\cline{2-11}
			component & 10\% & 20\% & 30\% & 40\% & 50\% & 10\% & 20\% & 30\% & 40\% & 50\% \\
			\hline
			$d2_x$ & 5\%   & 11\%  & 22\%  & 41\% & 87\% & -3\%  & -9\%  & -18\% & -32\% & -52\% \\
			$d2_y$ & -1\%  & 0\%   & 6\%   & 23\% & 81\% & 17\%  & 11\%  & 4\%   & -4\% & -17\% \\
			$r2_z$ & 6\%   & 15\%  & 27\%  & 47\% & 88\% & 16\%  & 10\%  & 1\%   & -11\% & -27\% \\
			$r5_z$ & 6\%   & 16\%  & 30\%  & 50\% & 89\% & 40\%  & 32\%  & 23\%  & 12\% & -3\% \\
			$r6_z$ & 7\%   & 19\%  & 33\%  & 53\% & 90\% & 32\%  & 24\%  & 16\%  & 6\% & -7\% \\
			$d3_x$ & 4\%   & 11\%  & 22\%  & 41\% & 86\% & -2\%  & -8\%  & -17\% & -31\% & -49\% \\
			$d3_y$ & -2\%  & 0\%   & 6\%   & 24\% & 81\% & 26\%  & 20\%  & 13\%  & 5\% & -7\% \\
			$r3_z$ & 4\%   & 12\%  & 24\%  & 44\% & 87\% & 23\%  & 16\%  & 8\%   & -3\% & -18\% \\
			\hline
			\#iter$_\textrm{L}$ & 4 & 5 & 7 & 13 & 65 \\
			\cline{1-6}
			\#iter$_\textrm{LU}$ & 3 & 4 & 6 & 8 & 14 \\
			\cline{1-6}
		\end{tabular}%
		\label{tab:frame_results}%
	\end{table}%
\end{example}

\section{Conclusions}
We have proposed and examined various preconditioning strategies, including double preconditioning based on the decomposition of the midpoint inverse, that aim to improve the numerical properties of interval parametric matrices in the context of solving interval parametric linear systems. We have proposed also a~new approach to solving interval parametric linear systems which employs the considered preconditioning strategies and revised affine forms. The numerical experiments have shown that the proposed approach enables us to solve an extended class of interval parametric linear systems. The obtained results indicate that double LU preconditioning is the most promising, i.e., it enables us to solve problems that cannot be solved by most existing methods for solving parametric interval linear systems. Moreover, it improves the results known in the literature, sometimes known as the only existing ones. It can be observed that the advantage of the double LU preconditioning increases with the increase of uncertainty and that it usually outperforms both SVD and QR based preconditionings. Based on the obtained results it can be concluded that the PKI$_\textrm{LU}$ method is useful for solving practical problems, however, the PHBR$_\textrm{LU}$ method is also suitable.



\bibliographystyle{spmpsci_uns}
\bibliography{lit}

\end{document}